\documentclass[journal]{IEEEtran}
\usepackage{amsmath}
\usepackage{amssymb}
\usepackage{amsthm}
\usepackage{mathrsfs}
\usepackage{graphicx}
\usepackage{mathtools}
\usepackage{algorithm}
\usepackage{algorithmic}
\usepackage{url}
\usepackage{verbatim}
\usepackage{color}
\usepackage{enumitem}
\usepackage[noadjust]{cite}
\usepackage{subfigure}
\usepackage[footnotesize]{caption}
\usepackage{eps2pdf}
\usepackage{stfloats}
\usepackage[normalem]{ulem}
\newcommand{\bx}{{\mathbf{x}}}
\DeclarePairedDelimiter\ceil{\lceil}{\rceil}
\newtheorem{theorem}{Theorem}

\newtheorem{proposition}[theorem]{Proposition}
\newtheorem{lemma}[theorem]{Lemma}
\newtheorem{remark}[theorem]{Remark}
\usepackage{array}

\allowdisplaybreaks[1]

\begin{document}
	
	\title{\vspace{-0.1cm} 
		Asynchronous Optimization over Graphs:    Linear Convergence under Error Bound Conditions
	}
	\author{Loris Cannelli, Francisco Facchinei$^\ast$, Gesualdo Scutari$^\ast$, and Vyacheslav Kungurtsev\vspace{-0.6cm}
		\thanks{$^\ast$Facchinei and Scutari contributed equally to this paper.\newline Cannelli is with Istituto Dalle Molle di studi sull'Intelligenza Artificiale (IDSIA), USI/SUPSI, Lugano, Switzerland, and with the School of Industrial Engineering, Purdue University West-Lafayette, IN, USA; email: \texttt{loris.cannelli@idsia.ch}. Scutari is with the School of Industrial Engineering, Purdue University, West-Lafayette, IN, USA; email: \texttt{gscutari@purdue.edu}. Facchinei is  with the Dept. of Computer, Control, and Management Engineering, University of Rome La Sapienza,  Rome, Italy; email: \texttt{francisco.facchinei@uniroma1.it}. Kungurtsev is with the Dept. of Computer Science, Czech Technical University in Prague, Prague, Czech Republic; email:
			\texttt{vyacheslav.kungurtsev@fel.cvut.cz}. \newline The work of  Cannelli and Scutari was supported by the USA NSF under Grants  CIF 1719205, and CMMI 1832688; and   the ARO under the Grant W911NF1810238. Kungurtsev was supported by the OP VVV project CZ.02.1.01/0.0/0.0/16\_019/0000765 ``Research Center for Informatics''.\vspace{-0.4cm}
		}}
		
		\IEEEpeerreviewmaketitle
		\maketitle
		\begin{abstract}
			We consider convex and nonconvex constrained optimization with a partially separable objective function:
			agents  minimize the sum of local objective functions, each of which  is  known  only by the associated agent  and   depends on the variables of that agent and those of  a few others.
			This partitioned setting arises in several applications of practical interest.
			We propose what is, to the best of our knowledge, the first distributed, asynchronous  algorithm with rate guarantees for this  class of problems. When the objective function is nonconvex, the algorithm provably converges  to a stationary solution at a sublinear rate whereas linear rate is achieved  
			under the renowned Luo-Tseng   error bound condition (which is less stringent than strong convexity).  Numerical results   on matrix completion and LASSO problems show the effectiveness of our method.
		\end{abstract}
		\begin{IEEEkeywords}
			Asynchronous algorithms; error bounds; linear rate; multi-agent systems; nonconvex optimization. \vspace{-0.3cm}
		\end{IEEEkeywords}
		\IEEEpeerreviewmaketitle
		\section{Introduction}
		\label{introduction}
		\IEEEPARstart{W}e study distributed, nonsmooth, nonconvex optimization   with 
		a partially separable sum-cost function. 
		Specifically,   consider a set of $N$ agents, each of them controlling/updating  a subset of the $n$ variables $\mathbf{x}\in\mathbb{R}^n$. Partitioning $\mathbf{x}=(\mathbf{x}_1^T,\ldots,\mathbf{x}_N^T)^T $,    $\mathbf{x}_i\in\mathbb{R}^{n_i}$ is the block of variables  owned by agent $i\in\mathcal{N}\triangleq\{1,\ldots,N\}$, with $\sum_i n_i =n$.
		All agents cooperatively aim at solving the following problem:
		\begin{equation}
			\begin{array}{rl}
				\underset{\mathbf{x}_i\in \mathcal{X}_i, i\in \mathcal{N}}{\min} & V(\mathbf{x}) \triangleq \underbrace{\sum\limits_{i=1}^Nf_i(\mathbf{x}_{\mathcal{N}_i})}_{\triangleq F(\mathbf{x})}+\underbrace{\sum\limits_{i=1}^Ng_i(\mathbf{x}_i)}_{\triangleq G(\mathbf{x})},
			\end{array}\tag{P}
			\label{problem}
		\end{equation}
		where  $\mathcal{N}_i$ denotes a small subset of $\mathcal{N}$ including the index $i$ and  {$\mathbf{x}_{\mathcal{N}_i}\triangleq[\mathbf{x}_j]_{j\in\mathcal{N}_i}$} denotes the column vector containing the blocks of  $\mathbf{x}$  indexed by $\mathcal{N}_i$;  $\mathcal{X}_i\subseteq \mathbb{R}^{n_i}$ is a closed  convex set; $f_i$ is a smooth (nonconvex) function that  depends only on   $\mathbf{x}_{\mathcal{N}_i}$; and $g_i$ is a convex (nonsmooth) function, instrumental to encode structural constraints on the solution,
		such as sparsity. Both $f_i$ and $g_i$ are assumed to be known only by agent $i$.
		
		The above  formulation   is motivated by a variety of applications of practical interest. For instance,
		 loss functions arising from  many machine learning problems  have the ``sparse'' pattern of $V$ in \eqref{problem}:    $n$ and $N$ are both very large but each  $f_i$ depends only on a   small number of components of $\mathbf{x}$, i.e., each subvector $\mathbf{x}_{\mathcal{N}_i}$  contains just a few components of $\mathbf{x}$. The same partitioned structure 	in  \eqref{problem} is suitable also to   model networked systems 	wherein agents are connected through a 	physical communication network and can communicate only with their immediate neighbors. In this setting, often $\mathcal{N}_i$ represents the set of neighbors of agent $i$ (including agent $i$ itself). Examples of such applications include
			 resource allocation problems and network utility maximization \cite{carli2013distributed}, state estimation in power networks \cite{kekatos2013distributed}, cooperative localization in wireless networks \cite{erseghe2012distributed}, and map building in robotic networks. Some concrete instances of Problem \eqref{problem} are discussed in Sec. \ref{prob_stat}. \vspace{-0.3cm}
		
	\subsection{ Major contributions}  \vspace{-0.1cm}
		 We focus on the design of distributed, asynchronous algorithms for \eqref{problem}, in the following sense:   i) Agents can update their block-variables at any time, without any coordination; and ii) when updating their own variables, agents can use a delayed out-of-sync information from the others.
		No constraint is imposed on the delay profiles: delays can be arbitrary, possibly  time-varying (but bounded). This model captures  several forms of asynchrony: some agents execute more iterations than others; some agents communicate more frequently than others; and inter-agent communications can be unreliable and/or subject to unpredictable, unknown, time-varying delays.
		
		While  several forms of asynchrony have been studied in the literature--see Sec. \ref{sec:related-works} for an overview of most relevant results--we are not aware of any distributed scheme that is compliant to the asynchronous  model (i)-(ii) and tailored to the \emph{partitioned} (nonconvex)  {distributed} formulation \eqref{problem}.  This paper fills this gap and proposes a general distributed, asynchronous algorithmic framework for   convex and nonconvex instances of \eqref{problem}.
	  The algorithm builds on Successive Convex Approximation (SCA) techniques:  agents solve asynchronously  [in the sense (i) and (ii) above] strongly convex approximations of the original problem \eqref{problem} by using (possibly) outdated information on the variables and  the gradients of the other agents.  		No specific activation mechanism for the agents' updates, coordination,  or communication protocol is assumed, but only some  mild conditions ensuring that  information used in the updates does not become infinitely old. For nonconvex instances of $V$, we prove that i)
	   every limit point of the sequence generated by the proposed asynchronous algorithm is a stationary solution  of \eqref{problem}; and ii) a suitable  measure of stationarity    vanishes at  a sublinear rate. When $V$ further satisfies  the  Luo-Tseng error bound condition \cite{luo1993error,luo1992linear}, both the sequence   and the objective value  converge at an R-linear rate ({when $V$ is nonconvex, convergence is to stationary solutions).}  This error bound condition is weaker than strong convexity and it is satisfied by a variety of    problems of interest, such as  LASSO, Group LASSO, and Logistic Regression, just to name a few (cf. Sec. \ref{sec:assumptions}). While linear convergence under error bounds  has been proved    for many {\it centralized} algorithms   \cite{luo1993error,tseng1991rate,tseng2009coordinate,zhang2013linear,drusvyatskiy2018error}, we are not aware of any such a result in the distributed setting; current works require  strong convexity to establish  linear rate of  synchronous and asynchronous {\it distributed }algorithms (see, e.g., \cite{Tian-ASYSONATA18,sun2019convergence,shi2015extra} and references therein).
 As a byproduct, our results provide also a  positive answer to the open question whether linear convergence could be proved for {\it distributed} asynchronous algorithms solving  highly dimensional empirical risk minimization problems,   such as LASSO and Logistic Regression, a fact that was empirically observed but, to our knowledge,  never proved.
	
	     \vspace{-0.3cm}

		\subsection{Related Works}\label{sec:related-works}    
		Since the seminal work \cite{tsitsiklis1986distributed}, asynchronous parallelism has been applied to several centralized solution methods, including (randomized) block-coordinate descent schemes \cite{tseng1991rate,tsitsiklis1986distributed,liu2015asyspcd,cannelli2019math,DavisEdmundsUdell,Bertsekas_Book-Parallel-Comp}, and stochastic gradient algorithms \cite{Hogwild!,lian2015asynchronous}. However, those methods  are not applicable to Problem \eqref{problem}
		, since they would require
  each agent to know
 the entire objective function $V$. 
		

		{\it Distributed} schemes exploring (some form of) asynchrony have been studied in \cite{iutzeler2013asynchronous,wei20131,bianchi2016coordinate,srivastava2011distributed,notarnicola2017asynchronous,xu2018convergence,notarnicola2016randomized,nedic2011asynchronous,wang2015cooperative,li2016distributed,tsianos2012distributed,tsianos2011distributed,lin2016distributed,doan2017impact,zhao2015asynchronous,kumar2017asynchronous,wu2018decentralized,peng2016arock,bof2017newton,hong2017distributed}; next,  we group them based upon the asynchrony features (i) and (ii).
		
		\noindent (a) \textit{Random activation and no delays} \cite{iutzeler2013asynchronous,wei20131,bianchi2016coordinate,srivastava2011distributed,notarnicola2017asynchronous,xu2018convergence,notarnicola2016randomized,nedic2011asynchronous,shah2018linearly}: While substantially different in the form of the updates performed by the agents, these schemes are all asynchronous in the sense of feature (i) only. Agents (or edge-connected agents)  are   randomly activated but, when performing their computations/updates, they must use the current information from their neighbors. This means that no form of delay is allowed. Furthermore, between two activations, agents must be in idle mode (i.e., able to continuously receive information).   Some form of coordination is thus needed to enforce the above conditions. 
All the schemes in this group but \cite{notarnicola2016randomized} can deal   with convex objectives only; and none of the above works   provide a convergence rate or complexity analysis.
		
	\noindent 	(b) \textit{Synchronous activation and delays} \cite{wang2015cooperative,li2016distributed,tsianos2012distributed,tsianos2011distributed,lin2016distributed,doan2017impact}: These schemes consider synchronous activation/updates of the agents, which can tolerate fixed computation delays (e.g., outdated gradient information) \cite{wang2015cooperative,li2016distributed}  or fixed \cite{tsianos2012distributed,doan2017impact} or time-varying \cite{tsianos2011distributed,lin2016distributed} communication delays. However   delays  cannot be arbitrary, but must be such that  no  loss can ever occur in the network: every agent's message must reach its intended destination within a finite  time interval. Finally, all  these algorithms are applicable only to convex   problems.
		
	\noindent 	(c) \textit{Random/cyclic activations and some form of delay} \cite{zhao2015asynchronous,kumar2017asynchronous,wu2018decentralized,peng2016arock,bof2017newton,hong2017distributed,zhu2018block,chang2016asynchronous,ma2018fast,zhang2014asynchronous}: These schemes allow for random \cite{zhao2015asynchronous,kumar2017asynchronous,peng2016arock,wu2018decentralized,zhu2018block} or deterministic uncoordinated \cite{bof2017newton,hong2017distributed,chang2016asynchronous,ma2018fast,jiang2019asynchronous,zhang2014asynchronous} activation of the (edge-based) agents, together with the presence of some form of delay in the updates/computations. Specifically, \cite{zhao2015asynchronous,kumar2017asynchronous,bof2017newton} can handle link failures--the information sent by an agent to its neighbors either gets lost or received with {\it no delay}--but cannot deal with other forms of delay (e.g., communication delays). In \cite{peng2016arock,wu2018decentralized,zhu2018block} a probabilistic model is assumed whereby agents are randomly activated and update their local variables using possibly delayed information. The model requires that the random variables modeling the activation of the agents are i.i.d and independent of the delay vector used by the agent to perform its update. While this assumption makes the convergence analysis possible, in reality there is a strong dependence of the delays on the activation index, as also noted by the same  authors  \cite{peng2016arock,wu2018decentralized}; see  \cite{cannelli2019math}  for a detailed discussion on this issue and some  counter examples. Closer to our setting are the asynchronous methods in   \cite{hong2017distributed,wu2018decentralized,Tian-ASYSONATA18,chang2016asynchronous,ma2018fast,jiang2019asynchronous,zhang2014asynchronous}. These models  however assume that    each function $f_i$   depends on the {\it entire} vector $\mathbf{x}$. As a consequence, a consensus mechanism on all the optimization variables is employed among the agents  at each iteration. Because of that, a direct application of  these consensus-based algorithms   to the partitioned formulation \eqref{problem} would lead to very inefficient schemes calling for      unnecessary   computation and communication overheads.  {Furthermore, the ADMM-like schemes \cite{zhu2018block,chang2016asynchronous,ma2018fast,jiang2019asynchronous,zhang2014asynchronous,hong2017distributed} can be implemented only on  very specific network architectures, such as star networks or hierarchical topologies with multiple  master and worker nodes.} Finally, notice that, with the exception of \cite{kumar2017asynchronous,hong2017distributed,Tian-ASYSONATA18,zhu2018block,chang2016asynchronous,ma2018fast,jiang2019asynchronous,zhang2014asynchronous} (resp. \cite{bof2017newton}), all these schemes are applicable to convex problems (resp. undirected graphs) only, with \cite{zhao2015asynchronous} further assuming that all the functions $f_i$ have the same minimizer.

		The rest of the paper is organized as follows: Section \ref{prob_stat} discusses  some motivating applications. 
		The proposed algorithm is introduced and analyzed in Section \ref{sec:algorithm}. Finally, numerical results are presented in Section \ref{numerical_results}.\vspace{-0.3cm}
		\section{Motivating Examples}\label{prob_stat}
		We discuss  next two instances of Problem \eqref{problem}, which will be  also   used in our numerical experiments to test our algorithms (cf.  Sec. \ref{numerical_results}). The first case study is the matrix completion problem--an example of large-scale nonconvex  empirical risk minimization. We show how to exploit the sparsity pattern in the data to rewrite the problem in the form   \eqref{problem}, so that efficient asynchronous algorithms levering multi-core architectures can be developed. The second example deals with   learning problems from networked data sets; in this setting  data are distributed across multiple nodes, whose communication network is modeled as a (directed) graph.  \\		\textit{Example \#1 --Matrix completion:} 
The matrix completion problem consists of estimating a low-rank   matrix $\mathbf{Z}\in\mathbb{R}^{M\times N}$ from a  subset $\Omega\subseteq \{1,\ldots, M\}\times \{1,\ldots ,N\}$ of its entries. Postulating   the low-rank factorization $\mathbf{Z}=\mathbf{X}^T\mathbf{Y}$, with
 $\mathbf{X}\in\mathbb{R}^{r\times M}$ and $\mathbf{Y}\in\mathbb{R}^{r\times N}$, the optimization problem reads \cite{srebro2005maximum}:\vspace{-0.1cm}
		\begin{equation}
				\underset{\overset{ \scriptstyle  \mathbf{X}\in\mathbb{R}^{r\times M}}{\scriptstyle \mathbf{Y}\in\mathbb{R}^{r\times N}}}{\min}
\;V(\mathbf{X},\mathbf{Y})\triangleq  \frac{1}{2}\left\|(\mathbf{X}^T\mathbf{Y}-\mathbf{Z})_{\Omega}\right\|_F^2
+\frac{\lambda}{2}\left\|\mathbf{X}\right\|_F^2+\frac{\xi}{2}\left\|\mathbf{Y}\right\|_F^2,
			\label{mc}\vspace{-0.1cm}
		\end{equation}
		where $\|\cdot\|_F$ is the Frobenius norm; $(\cdot)_{\Omega}$ is the projection operator, defined as  $[(\mathbf{X})_{\Omega}]_{(i,j)}=\mathbf{X}_{(i,j)}$, if $(i,j)\in\Omega$; and   $[(\mathbf{X})_{\Omega}]_{(i,j)}=0$ otherwise; and $\lambda,\xi>0$ are regularization parameters. In many applications,  the amount of data is so large that storage and processing from a single agent (e.g., core, machine) is not efficient or even feasible.   The proposed approach is then to leverage    multi-core machines by first casting \eqref{mc} in the form \eqref{problem}, and then employing  the parallel asynchronous framework developed in this paper.

		 Consider a distributed environment composed of $N$ agents, and assume that  the known entries $z_{mn}$, $(m,n)\in\Omega$, are partitioned among the agents. This partition along with the sparsity pattern of $(\mathbf{Z})_\Omega$   induce naturally the following   splitting of the optimization variables $
		 \mathbf{X}$ and $\mathbf{Y}$ across the agents.  Let $\mathbf{x}_m$ and $\mathbf{y}_n$ denote the $m$-th and  the $n$-th column of $
		 \mathbf{X}$ and $\mathbf{Y}$, respectively;  the agent owning
		 $z_{mn}$ will control/update the variables  $\mathbf{x}_m$ (or $\mathbf{y}_n$), and it is connected to the agent that
		 optimizes the column $\mathbf{y}_n$ (or $\mathbf{x}_m$). By doing so,   we  minimize the overlapping across the block-variables  and, consequently, the communications among the agents. Problem \eqref{mc}  can be then rewritten in the  multi-agent form \eqref{problem}, setting \vspace{-0.1cm}
		 \begin{equation}f_i((\mathbf{X},\mathbf{Y})_{\mathcal{N}_i})=\frac{1}{2}\,\sum\limits_{(m,n)\in\Omega_i}(\mathbf{x}_m^T\,\mathbf{y}_n-z_{mn})^2\vspace{-0.3cm}\end{equation} and \vspace{-0.1cm}\begin{equation}g_i\left(\{\mathbf{x}_m\}_{m\in X_i},\{\mathbf{y}_n\}_{n\in Y_i}\right) = \frac	{\lambda}{2}\sum\limits_{m\in X_i}\|\mathbf{x}_m\|^2_2+\frac{\xi}{2}\sum\limits_{n\in Y_i}\|\mathbf{y}_n\|
			_2^2,\end{equation} where $\Omega_i\subseteq \Omega$ contains the indices associated to the components of    $(\mathbf{Z})_\Omega$ owned by agent $i$, and $X_i$ (resp. $Y_i$) is the set of the column indexes of $
		\mathbf{X}$ (resp.  $\mathbf{Y}$)  controlled by agent $i$.

		\noindent
		\textit{Example \#2 -- Empirical risk minimization over networks: } Consider now a network  setting where data are distributed across  $N$ geographically separated nodes. As concrete example,  let us pick   the renowned LASSO problem \cite{tibshirani1996regression}:
		\begin{equation}
			\begin{array}{rl}
				\underset{\mathbf{x}=(\mathbf{x}_1^T,\ldots,\mathbf{x}^T_N)^T\in \mathbb{R}^n}{\min} & \|\mathbf{Ax}-\mathbf{b}\|_2^2+\lambda\|\mathbf{x}\|_1,
			\end{array}
			\label{lasso}
		\end{equation}
		where $\mathbf{A}\in\mathbb{R}^{m\times n}, \mathbf{b}\in\mathbb{R}^m$, and $\lambda>0$ is a regularization parameter. 
		Note that \eqref{lasso} easily falls into Problem \eqref{problem}; for each   $i\in\mathcal{N}$, it is sufficient to set $f_i(\mathbf{x})=\|\mathbf{A}_i\mathbf{x}+\mathbf{b}_i\|_2^2$, with $\mathbf{A}_i\in\mathbb{R}^{m\times n}$ and $\mathbf{b}_i\in\mathbb{R}^m$ such that $\mathbf{A}=\sum_{i=1}^N\mathbf{A}_i$ and $\mathbf{b}=\sum_{i=1}^N\mathbf{b}_i$; and $g_i=\|\mathbf{x}_i\|_1$.  $\mathbf{A}_i$ and $\mathbf{b}_i$ represent in fact the  data stored at agent $i$'s side. Under specific sparsity patterns in the data, 
		the local matrices $\mathbf{A}_i$ may be (or constructed to be)   such that each local function $f_i$  depends only on some of the block variables $\mathbf{x}_i$. These dependencies will   define the sets $\mathcal{N}_i$ associated to  each agent $i$.  Note that $\mathcal{N}_i$ need not coincide with  the   neighbors  of agent $i$ in the communication network (graph).  That is, the graph modeling the dependence across the block-variables--the one with  node set  $\mathcal{N}$ and edge set  $\mathcal{E}=\{(i,j): j\in \mathcal{N}_i,\text{for some}\,i\!\in\!\mathcal{N}\}$--might not coincide with the communication graph. This can be desirable, e.g., when the communication graph is populated by inefficient communication links, which one  wants to avoid using.\vspace{-0.1cm}

		\section{Distributed Asynchronous Algorithm}		
		\label{sec:algorithm}
		In the proposed asynchronous model,  agents  update their block-variables    without any coordination. Let $k$ be the iteration counter: the iteration  $k\rightarrow k+1$ is triggered  when   {one} agent, say $i$,  updates its own block   $\mathbf{x}_i$ from $\mathbf{x}_i^k$ to $\mathbf{x}_i^{k+1}$. Hence,   $\mathbf{x}^k$ and $\mathbf{x}^{k+1}$ only differ in the $i$-th block $\mathbf{x}_i$.
To perform its update, agent $i$ minimizes a
strongly convex approximation of  $\sum_{j\in\mathcal{N}_i}f_j-$the part of $V$ that depends on $
\mathbf{x}_i-$using possibly outdated information collected from the other agents $j\in \mathcal{N}_i$.
To represent this situation,
		let  $\mathbf{x}_j^{k-d_j^k(i,i)}$, $j\in\mathcal{N}_{i}\backslash\{i\}$,
		denote the estimate held by agent $i$ of  agent $j$'s variable $\mathbf{x}_j^k$, where  $d^k_j(i,i)$ is a
		nonnegative (integer) delay (the reason for the double index $(i,i)$ in $d^k_j$ will become clear shortly). If   $d^k_j(i,i)=0$,
		agent
		$i$ owns the most recent information on the variable of agent $j$,
		otherwise  $\mathbf{x}_j^{k-d^k_j(i,i)}$ is some delayed version of  $\mathbf{x}_j^{k}$. We define as $\mathbf{d}^k(i,i)\triangleq [d_l^k(i,i)]_{l\in\mathcal{N}_i}$ the \textit{delay vector} collecting these delays; for ease of notation $\mathbf{d}^k(i,i)$ contains also the value $d^k_i(i,i)$, set to zero, as 
		each agent has always access to current values of its own  variables.		Using the above notation and  {recalling that $f_i$ depends on $\mathbf{x}_{
				\mathcal{N}_i} $,} agent $i$ at iteration $k$ solves   the
		following strongly convex  subproblem:\vspace{-0.1cm}
		\begin{align}
			\label{best_resp} & {\widehat{\mathbf{x}}_i^k\!\triangleq\underset{\mathbf{x}_i\in\mathcal{X}_i}{\arg\min}\,
\Bigg\{\tilde{f}_i\left(\mathbf{x}_i;\mathbf{x}_{
				\mathcal{N}_i}^{k-\mathbf{d}^k(i,i)}\right)+}
			\\\nonumber&\sum\limits_{j\in\mathcal{N}_i\backslash\{i\}}\!\!\!\!\left\langle \nabla_{\mathbf{x}_i}f_j\left(\mathbf{x}_{\mathcal{N}_j}^{k-
				\mathbf{d}^k(i,j)}\right),\mathbf{x}_{i}\!-\!\mathbf{x}_i^k\right\rangle + g_{i}(\mathbf{x}_i)\Bigg\},
		\end{align}
 where we defined    $\mathbf{x}_{\mathcal{N}_j}^{k-
			\mathbf{d}^k(i,j)}\triangleq[\mathbf{x}_l^{k-d^k_l(i,j)}]_{l\in\mathcal{N}_j}$,	  $j\in\mathcal{N}_i$.	

The term $\tilde{f}_i$ in \eqref{best_resp} is a  strongly convex surrogate that replaces the nonconvex
		function $f_i$ known by agent $i$; an outdated value of the variables of the other agents is used, $\mathbf{x}_{
				\mathcal{N}_i}^{k-\mathbf{d}^k(i,i)}$, to build this function.  Examples of valid surrogates are discussed in Sec. \ref{sec:assumptions}.
The second term in \eqref{best_resp}  approximates
			$\sum_{j\in			\mathcal{N}_i\backslash\{i\}} f_j$
by replacing  each  $f_j$  by its first order approximation at (possibly outdated) $\mathbf{x}_{\mathcal{N}_j}^{k-
				\mathbf{d}^k(i,j)}$ (with
			$\nabla_{\mathbf{x}_i}f_j$ denoting the gradient of $f_j$ with respect to the block $\mathbf{x}_i$), where $\mathbf{d}^k(i,j)\triangleq\left[d^k_l(i,j)\right]_{l\in\mathcal{N}_j}$, with $d_l^k(i,j)\geq0$ representing the delay of the information that $i$ knows about the gradient $\nabla_{\mathbf{x}_i}f_j$. This source of delay  {on the gradients} is due to two facts, namely:   i) agents $j \in\mathcal{N}_i\setminus \{i\}$ may communicate to $i$ its gradient $\nabla_{\mathbf{x}_i}f_j$ occasionally; and ii)  $\nabla_{\mathbf{x}_i}f_j$ is  generally  computed at some outdated point, as agent $j$ itself  may not  have access of the last information of the variables of the agents in $\mathcal{N}_j\setminus \{j\}$.

							
Once 		${\widehat{\mathbf{x}}_i^k}$ has been computed, agent $i$   sets
		\begin{equation} \mathbf{x}_i^{k+1}=\mathbf{x}^k_i+\gamma\left({\widehat{\mathbf{x}}_i^k}-\mathbf{x}_i^k\right),
			\label{update}
		\end{equation}
		where $\gamma\in(0;1]$ is suitably chosen   stepsize  (cf. Sec. \ref{sec:assumptions}).		
		
		The proposed distributed asynchronous algorithm, termed Distributed Asynchronous FLexible ParallEl Algorithm (DAsyFLEXA), is formally described in  Algorithm~\ref{algorithm1}. We set $\mathbf{x}^t_i=\mathbf{x}^0_i$, for all $t<0$ and $i\in\mathcal{N}$, without loss of generality.
		\begin{algorithm}
			\caption{\textbf{Distributed Asynchronous FLexible ParallEl Algorithm (DAsyFLEXA)}}
			\label{algorithm1}
			\begin{algorithmic}
				\STATE{\textbf{Initialization:} \hspace{-.1cm}$k\hspace{-.1cm}=\hspace{-.1cm}0$; $\mathbf{x}^0\hspace{-.1cm}\in\hspace{-.1cm}\mathcal{X}\hspace{-.1cm}\triangleq\hspace{-.1cm}\prod_i\mathcal{X}_i$;  $\mathbf{x}^t\hspace{-.1cm}=\hspace{-.1cm}\mathbf{x}^0$, $t\hspace{-.1cm}<\hspace{-.1cm}0$; $\hspace{-.1cm}\gamma\in(0;1]$.}\\
				\WHILE{ a termination criterion is not met}
				\STATE{\texttt{(S.1)}: Pick agent $i^k$ and delays $\{\mathbf{d}^k(i^k,j)\}_{j\in\mathcal{N}_{i^k}}$;}
				\STATE{\texttt{(S.2)}: Compute $\widehat{\mathbf{x}}_{i^k}^k$ according to \eqref{best_resp};}
				\STATE{\texttt{(S.3)}: Update $\mathbf{x}_{i^k}^k$ according to \eqref{update};}
				\STATE{\texttt{(S.4)}: Update the global iteration counter $k \leftarrow k+1$;}
				\ENDWHILE
			\end{algorithmic}
		\end{algorithm}

	 We stress that agents need  know neither   the iteration counter $k$ nor the vector of delays.
		No one  ``picks agent $i^k$ and the delays $\{\mathbf{d}^k(i^k,j)\}_{j\in\mathcal{N}_{i^k}}$'' in   \texttt{(S.1)}. This is just an {\em a posteriori} view of the algorithm dynamics: all agents asynchronously and continuously collect information from their neighbors and use it to update $\mathbf{x}_i$; when one agent has completed an update the iteration index $k$ is increased and $i^k$ is defined.

 		
		 	 \vspace{-0.4cm}
		
		\subsection{Assumptions}\label{sec:assumptions}
		Before studying convergence of  Algorithm \ref{algorithm1}, we state the main assumptions on Problem \eqref{problem} and  the algorithmic choices.

		\noindent\textbf{On Problem \eqref{problem}.}
		Below,  we will use the following  conventions: When  a function is said to be  differentiable on a certain domain, it is understood  that the function is differentiable on an open set containing the domain. 
		We say that  $f_i$ is block-$LC^1$ on a set if it is continuously differentiable on that set and  $\nabla_{\mathbf{x}_j}f_i$  are locally Lipschitz.  We say $V$ is {\em coercive} on $\mathcal{X} =\prod_{i}\mathcal{X}_i$, if $\lim\limits_{\substack{\|\mathbf{x}\|\to +\infty, \mathbf{x} \in{\cal X}}} V(\mathbf{x}) = +\infty$; this is equivalent to requiring that all level sets of $V$ in $\cal X$ are compact.

		\noindent \textbf{Assumption A (On Problem \eqref{problem}):}
		\begin{description}
			\item[(A1)]  Each set $\mathcal{X}_i\subseteq\mathbb{R}^{n_i}$ is nonempty, closed, and convex;
			\item[(A2)]  At least one of the following conditions is satisfied
			
			(a) ${\cal L}^0 \triangleq \{\mathbf{x} \in {\cal X}: V(\mathbf{x})\leq V(\mathbf{x}^0)\}$ is compact and all  $f_i$ are block-$LC^1$ on ${\mathcal{X}_{\mathcal{N}_i}\triangleq\underset{j\in\mathcal{N}_i}{\Pi}\mathcal{X}_j}$;
			
			(b) 	 All $f_i$ are $C^1$ and their   gradients    $\nabla_{\mathbf{x}_j}f_i$, $j \in {\cal N}_i$, are {\em globally} Lipschitz on
			 {$\mathcal{X}_{\mathcal{N}_i}$};

			
			\item[(A3)] Each $g_i:\mathcal{X}_i\rightarrow\mathbb{R}$ is convex;
			\item[(A4)] Problem \eqref{problem} has a solution;
			\item[(A5)] The communication graph $\mathcal{G}$ is   connected.
		\end{description}
		
		\indent The above assumptions are standard and satisfied by many practical problems. For instance,  A2(a) holds if $V$ is coercive on $\cal X$ and  all  $f_i$ are block-$LC^1$ on  {$\mathcal{X}_{\mathcal{N}_i}$}. Note that  Example \#2 satisfies A2(b); A2(a) is motivated by applications such as Example \#1, which do not satisfy A2(b).   A3 is a  common assumption in the literature of parallel and distributed methods for the class of problems \eqref{problem}; two renowned examples are $g_i(\mathbf{x}_i)=\|\mathbf{x}_i\|_1$ and $g_i(\mathbf{x}_i)=\|\mathbf{x}_i\|_2$. Finally,  A4 is satisfied if, for example, $V$ is coercive or if $\cal X$ is bounded.  \vspace{-0.2cm}
	\begin{remark}   Extensions to the case of directed graphs or the case where each agent updates multiple block-variables are easy, but not discussed here for the sake of simplicity.\vspace{-0.1cm}\end{remark}
		

			The aim of Algorithm \ref{algorithm1} is to find {\em stationary solutions} of   \eqref{problem}, i.e.  points  {$\mathbf{x}^\star\in
		\mathcal{X}$  such that}\vspace{-0.1cm}
		\begin{equation*}
		\left\langle \nabla F(\mathbf{x}^\star) + \boldsymbol{\xi},\mathbf{y} - \mathbf{x}^\star\right\rangle+G(\mathbf{y})-G(\mathbf{x}^\star)\geq 0, \qquad \forall \mathbf{y} \in \mathcal{X}.
		\end{equation*}
		Let
		$\mathcal{X}^\star\subseteq \mathbb{R}^n$ denote the set of  such stationary solutions.
		
		
		\noindent\textbf{On an error bound condition.}
		We prove     linear convergence  of Algorithm~\ref{algorithm1} under the Luo-Tseng  error bound condition, which is stated next.  Recall the definition: 
		given $\alpha>0$, 
		\begin{equation*}
			\texttt{prox}_{\alpha G}(\mathbf{z})\triangleq\underset{\mathbf{y}\in\mathcal{X}}{\arg\min}\,\left\{\alpha G(\mathbf{y})+\frac{1}{2}\|\mathbf{y}-\mathbf{z}\|_2^2\right\}.
		\end{equation*}
			Furthermore, given $\mathbf{x}\in \mathbb{R}^n$, let  		\begin{equation*}
			d(\mathbf{x}, \mathcal X^\star)\triangleq \underset{\mathbf{x}^\star\in\mathcal{X}^\star}{\min} \|\mathbf{x}-\mathbf{x}^\star\|_2,\qquad P_{\mathcal{X}^\star}(\mathbf{x})\triangleq \underset{\mathbf{x}^\star\in\mathcal{X}^\star}{\text{argmin}} \|\mathbf{x}-\mathbf{x}^\star\|_2.
		\end{equation*}
		Note that $P_{\mathcal{X}^\star}(\mathbf{x})\neq \emptyset$, as $\mathcal X^\star$ is closed.
		
		\noindent\textbf{Assumption B (Luo-Tseng error bound):}\begin{description}
			\item[(B1)] For any $\eta>\underset{\mathbf{x}\in\mathcal{X}}{\min}\,V(\mathbf{x})$, there exist   $\epsilon,\kappa>0$ such that: $$\hspace{-1cm}\left.\begin{array}{l}
			V(\mathbf{x})\leq\eta,\\
			\|\mathbf{x}-\texttt{prox}_{G}\left(\nabla F(\mathbf{x})-\mathbf{x}\right)\|_{2}\leq\epsilon
			\end{array}\right\} \Rightarrow$$$$d(\mathbf{x},\mathcal{X}^{\star})\leq\kappa\left\Vert \mathbf{x}-\texttt{prox}_{G}\left(\nabla F(\mathbf{x})-\mathbf{x}\right)\right\Vert_{2};$$
			\item[(B2)]   There exists  $\delta>0$ such that $$\left.\begin{array}{l}
			\mathbf{x},\mathbf{y}\in\mathcal{X}^{*},\\
			V(\mathbf{x})\neq V(\mathbf{y})
			\end{array}\right\}  \Rightarrow \|\mathbf{x}-\mathbf{y}\|_{2}\geq\delta.$$
			
		\end{description}
		B1 is a local Lipschitzian error bound:  the distance of $\bx$ from  $\mathcal{X}^\star$ is of
		the same order of the norm of the residual  $\mathbf{x}-\texttt{prox}_{G}\left(\nabla F(\mathbf{x})-\mathbf{x}
		\right)$ at $\bx$. It is not difficult  to check, that  $\bx\in \mathcal{X}^\star$ if and only if
		$\mathbf{x}-\texttt{prox}_{G}\left(\nabla F(\mathbf{x})-\mathbf{x}\right)=0$.
		Error bounds of this kind have been extensively studied in the literature; see \cite{luo1993error,tseng2009coordinate} and references
		therein. Examples of problems satisfying Assumption B include:    LASSO, Group LASSO, Logistic Regression, unconstrained optimization with smooth nonconvex quadratic objective or    $F(\mathbf{A}\mathbf{x})$, with $F$ being strongly convex and $\mathbf{A}$ being arbitrary.  
		 B2 states that the level curves of $V$ restricted to   $\mathcal{X}^\star$ are
		``properly separated''.   B2 is trivially satisfied, e.g.,  if $V$ is convex, if $\mathcal{X}$ is bounded, or if   \eqref{problem} has a finite number of stationary solutions.
		
		
		\smallskip
		
		\noindent\textbf{On the subproblems \eqref{best_resp}.}
		The surrogate  functions $\tilde f_i$     satisfy  the following fairly standard conditions  ($\nabla \tilde{f}_i$   denotes the partial gradient of $\tilde{f}_i$ w.r.t. the first argument).
		
		\noindent \textbf{Assumption C} Each $\tilde f_i:\mathcal{X}_i\times  \mathcal{X}_{\mathcal{N}_i}\to \mathbb{R}$ is chosen so that\vspace{-0.cm}
		\begin{description}
			\item[ (C1)] $\tilde f_{i} (\mathbf{\cdot}; \mathbf{y})$ is $C^1$ and $\tau$-strongly convex on $ \mathcal{X}_i$, for all $\mathbf{y}\in \mathcal{X}_{\mathcal{N}_i}$;\smallskip
			\item[  (C2)]   $\nabla \tilde f_{i} (\mathbf{y}_i;\mathbf{y}_{\mathcal{N}_i}) = \nabla_{\mathbf{y}_i}f_i(\mathbf{y}_{\mathcal{N}_i})$, for all $\mathbf{y}\in  \mathcal{X}$; \smallskip
			\item[  (C3)]  $\nabla \tilde f_{i} (\mathbf{y};\cdot)$	is $L_i$-Lipschitz continuous	on $ \mathcal{X}_{\mathcal{N}_i}$,	for all $\mathbf{y}\in  \mathcal{X}_i$.
		\end{description}
		A wide array of   surrogate functions $\tilde f_{i}$ satisfying Assumption C can be found in  \cite{FLEXA}; three examples are discussed next.
		
		\noindent $\bullet$ It is always possible to choose $\tilde f_{i}$  as the first-order approximation of $f_i$:   $\tilde f_{i} (\mathbf{x}_i; \mathbf{y}_{\mathcal{N}_i})
		= \left\langle\nabla_{{\mathbf{x}}_i}   f(\mathbf{y}_{\mathcal{N}_i}),\bx_i - \mathbf{y}_i\right\rangle + c\| \bx_i - \mathbf{y}_i \|^2_2$, where $c$ is a positive constant.
		
		\noindent $\bullet$ If $f_i$ is block-wise uniformly convex, instead of linearizing $f_i$ one can exploit a second-order approximation and set  $\tilde{f}_i(\mathbf{x}_i;\mathbf{y}_{\mathcal{N}_i})=f_i(\mathbf{y}_{\mathcal{N}_i})+\left\langle\nabla_{\mathbf{x}_i}   f_i(\mathbf{y}_{\mathcal{N}_i}),\bx_i - \mathbf{y}_i\right\rangle+\frac{1}{2}(\mathbf{x}_i-\mathbf{y}_i)^T\nabla^2_{\mathbf{x}_i\mathbf{x}_i}f_i(\mathbf{y}_{\mathcal{N}_i})(\mathbf{x}_i-\mathbf{y}_i)+c\|\mathbf{x}_i-\mathbf{y}_i\|^2_2$, for any $\mathbf{y}\in\mathcal{X}$, where $c$ is a positive constant.
		
		\noindent$\bullet$ In the same setting as above, one can also better preserve the partial convexity of $f_i$ and choose  $\tilde{f}_i(\mathbf{x}_i;\mathbf{y}_{\mathcal{N}_i})=f_i(\mathbf{x}_i,\mathbf{y}_{\mathcal{N}_i\backslash\{i\}})+c\|\mathbf{x}_i-\mathbf{y}_i\|^2_2$, for any $\mathbf{y}\in\mathcal{X}$. 
		\smallskip

\noindent\textbf{On the asynchronous/communication model.}
The way  agent $i$  builds its own
		estimates $\mathbf{x}_{\mathcal{N}_i}^{k-\mathbf{d}^k(i,i)}$ and $\nabla_{\mathbf{x}_i}f_j\left(\mathbf{x}_{\mathcal{N}_j}^{k-\mathbf{d}^k(i,j)}\right)$,
		$j\in\mathcal{N}_i\backslash\{i\}$,  depends on the particular asynchronous model and communication
		protocol under consideration and it is immaterial to the convergence of Algorithm~\ref{algorithm1}. This is  a major departure from   previous works, such as  \cite{notarnicola2016randomized,iutzeler2013asynchronous,bianchi2016coordinate},  which instead enforce specific asynchrony and communication protocols. We only require the following mild conditions.   

		\noindent \textbf{Assumption D (On the asynchronous model):}
		\begin{description}
			\item[(D1)]  Every block  variable of $\mathbf{x}$ is  updated at most every $B\geq N$ iterations, i.e., $
			\cup_{t=k}^{k+B-1} \,i^t=\mathcal{N}$, for all $k$;
			\item[(D2)]  {$\exists\,D\in[0,B]$, such that every component of $\mathbf{d}^k(i,j)$, $i\in\mathcal{N}$, $j\in\mathcal{N}_i$, is not greater than $D$, for any $k\geq0$.}\footnote{While $\texttt{(S.2)}$ in Algorithm 1 is defined once $\mathbf{d}^k(i^k,j)$, $j\in \mathcal{N}_{i^k}$ is given, here we extend the definition of  the delay vectors   $\mathbf{d}^k(i,j)$ to all $i,j\in \mathcal{N}$, whose values are set to the delays of the information known by the associated   agent  on the variables  {and   gradients} of the others, at the time agent $i^k$ performs its update. This will simplify the notation in some of the technical derivations. }
		\end{description}
		Assumption D is satisfied   virtually in all practical  scenarios.  D1 controls the frequency of the
		updates and  is satisfied, for example, by any \textit{essentially cyclic rules}. In practice, it   is automatically satisfied, e.g., if each agent wakes up and performs an
		update whenever some internal clock ticks, without any   centralized coordination.
		D2 imposes a mild condition on the  communication protocol employed by the agents:  information used in the agents' updates can not become infinitely old. While this implies agents communicate sufficiently often, it does not enforce  any specific protocol on the activation/idle time/communication. For instance,  differently from several asynchronous schemes in the literature
\cite{notarnicola2016randomized,iutzeler2013asynchronous,wei20131,bianchi2016coordinate,srivastava2011distributed,nedic2011asynchronous,zhao2015asynchronous},    agents need not be  always  in ``idle mode''   to continuously    receive messages from their
 neighbors.
		Notice   that time varying delays satisfying D2 model also   packet losses. 		
		 \vspace{-0.3cm}

		\subsection{Convergence Analysis}
		\label{convergence}
		We are now in the position to state the main convergence results for DAsyFLEXA.
	 {For nonconvex instances of  \eqref{problem},   an appropriate measure of optimality is needed to evaluate the progress of the algorithm towards stationarity. In order to define such a   measure, we first introduce the following quantities: for any $k\geq0$  {and $i\in\mathcal{N}$},\vspace{-0.2cm}
		\begin{align}
\label{def bar x} & {\widehat{\bar{\mathbf{x}}}_i^k\!\triangleq\underset{\mathbf{x}_i\in\mathcal{X}_i}{\arg\min}\,
	\Bigg\{\tilde{f}_i\left(\mathbf{x}_i;\mathbf{x}_{
		\mathcal{N}_i}^{k}\right)+}
\\\nonumber&\sum\limits_{j\in\mathcal{N}_i\backslash\{i\}}\!\!\!\!\left\langle \nabla_{\mathbf{x}_i}f_j\left(\mathbf{x}_{\mathcal{N}_j}^{k}\right),\mathbf{x}_{i}\!-\!\mathbf{x}_i^k\right\rangle + g_{i}(\mathbf{x}_i)\Bigg\},
\end{align}
 {where $\widehat{\bar{\mathbf{x}}}_i^k$ is a ``synchronous'' instance of $\widehat{\mathbf{x}}^k_i$ [cf. \eqref{best_resp}] wherein all  $\mathbf{d}^k(i,j)=\mathbf{0}$.}
		Convergence to stationarity  is monitored by the following merit function:
		\begin{equation}
			 {M_V(\mathbf{x}^k)\triangleq\|\widehat{\bar{\mathbf{x}}}^k-\mathbf{x}^k\|_2^2},\quad\text{with}\quad \widehat{\bar{\mathbf{x}}}^k\triangleq\left[\widehat{\bar{\mathbf{x}}}_i^k\right]_{i\in\mathcal{N}}.
		\end{equation}
	Note that $M_V$ is a valid measure of stationarity, as $M_V$ is continuous and $M_V(\mathbf{x}^k)=0$ if and only if $\mathbf{x}^k\in \mathcal{X}^\star$.

		The following theorem shows that, when   agents use a sufficiently small stepsize, the sequence of the iterates produced by DAsyFLEXA converges to  a stationary solution of \eqref{problem}, driving
$M_V(\mathbf{x}^k)$	 to zero at a sublinear  rate.
In the theorem we use two positive constants, $L$ and $C_1$, whose definition is given in  Appendix \ref{technical} and \ref{step3} [cf. \eqref{C1}], respectively. Suffices to say, here, that $L$ is essentially a Lipschitz constant for the partial gradients $\nabla_{\mathbf{x}_i}f_i$  whose definition varies according to whether A2(a) or A2(b) holds. In the latter case,   $L$ is simply the largest global Lipschitz constant for all $\nabla_{\mathbf{x}_j} f_i$'s. In the former case,
	the sequences $\{\mathbf{x}^k\}$ 
	 and  {$\{\widehat{\mathbf{x}}^k\}$,  {with} $\widehat{\mathbf{x}}^k\triangleq\left[\widehat{\mathbf{x}}_i^k\right]_{i\in\mathcal{N}}$,} are proved to be  bounded [cf. Theorem \ref{compl}(c)];  $L$ is then the Lipschitz constant of  all $\nabla_{\mathbf{x}_j} f_i$'s over the  compact set confining these sequences.
		\begin{theorem}\label{compl}			
			Given Problem \eqref{problem} under Assumption A; let $\{\mathbf{x}^k\}$  be the
			sequence generated by DAsyFLEXA, under Assumptions C, and D. Choose $\gamma\in(0,1]$ such  that  $\gamma<\frac{2\tau}{L\left(2+\rho^2D^2\right)}$, with $\rho\triangleq
			{\max_{i\in\mathcal{N}}}\,|\mathcal{N}_i|$.
			Then, there hold:
			\begin{itemize}
				\item[(a)] Any limit point of $\{\mathbf{x}^k\}$ is a stationary solution of  \eqref{problem};
				\item[(b)] In at most $T_\epsilon$ iterations, DAsyFLEXA drives the stationarity measure $M_V(\mathbf{x}^k)$ below $\epsilon$, $\epsilon>0$, where
				\begin{align*}
					T_\epsilon= \ceil*{ {C_1\left(V(\mathbf{x}^0)-\underset{\mathbf{x}\in\mathcal{X}}{\min}\,V(\mathbf{x})\right)}\cdot\frac{1}{\epsilon}},
				\end{align*}	
				where $C_1>0$ is a constant defined in Appendix \ref{step3} [cf. \eqref{C1}], which depends on $\rho, L_i, i\in\mathcal{N}, L, \tau, \gamma, N, B$, and $D$.		
				%
				
				\item[(c)] If, in particular A2(a) is satisfied, $\{\mathbf{x}^k\}$ is bounded.
			\end{itemize}
			
				\end{theorem}	
		\begin{proof}
			See Appendix \ref{proof_general}.\vspace{-0.2cm}
		\end{proof}
		  Theorem \ref{compl}	provides a unified set of convergence conditions for several algorithms,  asynchronous models and communication protocols.
		Note that when   $D =0$,
		the condition on $\gamma$, reduces to the renowned condition  used in the synchronous proximal-gradient algorithm.   The term $D^2$ in the denominator of the upper-bound on $\gamma$ should then  be seen as the price to pay for asynchrony: the larger the possible delay $D$,  the smaller $\gamma$,   to make the algorithm robust to asynchrony/delays.

		Theorem \ref{linear} improves on the convergence of DAsyFLEXA, when $V$  satisfies the error  bound condition in Assumption B. Specifically, convergence of the whole sequence
		$\{ \mathbf{x}^k\}$   to a stationary solution $\mathbf{x}^\star$ is established (in contrast with subsequence  convergence   in Theorem \ref{compl} (b)], and  suitable subsequences that converge \emph{linearly} are  identified.\vspace{-0.2cm}
		\begin{theorem}\label{linear}
			Given Problem \eqref{problem} under Assumptions A and B, let $\{\mathbf{x}^k\}$  be the
			sequence generated by DAsyFLEXA, under  Assumptions C and D. Suppose that  $ {\gamma}/{\tau}>0$ is sufficiently small. Then, $\{\mathbf{x}^{t+kB}\}$ and $\{V(\mathbf{x}^{t+kB})\}$,   $t\in\{0,\ldots, B-1\}$, converge at least R-linearly  to  {some} $\mathbf{x}^\star{\in\mathcal{X}^\star}$ and $V^\star\triangleq V(\mathbf{x}^\star)$, respectively, that is
			\begin{equation*}
				V(\mathbf{x}^{t+kB})-V^\star=\mathcal{O}\left(\lambda^{t+kB}\right),
			\end{equation*}
			\begin{equation*}
				\|\mathbf{x}^{t+kB}-\mathbf{x}^\star\|=\mathcal{O}\left(\sqrt{\lambda^{t+kB}}\right),
			\end{equation*}
			where $\lambda\in(0;1)$ is a constant defined in Appendix \ref{proof_linear} [cf. \eqref{lambdalab}], which depends on $\rho, L_i, i\in\mathcal{N}, L, \tau, \gamma, N, B$, and $D$.\vspace{-0.2cm}
		\end{theorem}	
		\begin{proof}
			See Appendix \ref{proof_linear}.
		\end{proof}
		In essence,  the theorem proves  a $B$-steps linear convergence rate.
To the best of our knowledge,    this is the first (linear) convergence rate  result in the literature for an asynchronous algorithm in the setting considered in this paper.\vspace{-0.2cm}
\section{Numerical Results}
\label{numerical_results}
In this section we report some numerical results on the two problems described in Section \ref{prob_stat}.
 {All our experiments were run} on the Archimedes1 cluster computer at Purdue University, equipped with two 22-cores Intel E5-2699Av4 processors (44 cores in total) and 512GB of RAM. {Code for the LASSO problem was written in MATLAB R2019a; code for the Matrix Completion problem was written in C++ using the OpenMPI library for parallel and asynchronous operations.}\vspace{-0.3cm}

\subsection{Distributed LASSO}
\noindent {\textbf{Problem setting}.} We simulate the  (convex) LASSO problem   stated in \eqref{lasso}.
The underlying sparse linear model is generated as follows: $\mathbf{b} = \mathbf{A}{\mathbf{x}^\star}+\mathbf{e}$,
where $\mathbf{A}{\in\mathbb{R}^{15000\times30000}}$. $\mathbf{A}$, ${\mathbf{x}^\star}$ and $\mathbf{e}$ have i.i.d. elements, drawn
from a Gaussian $\mathcal{N}(0,\sigma^2)$ distribution, with $\sigma=1$ for $\mathbf{A}$
and ${\mathbf{x}^\star}$, and $\sigma=0.1$ for the noise vector $\mathbf{e}$. Entries of  $\mathbf{A}$ are then normalized by  $\|\mathbf{A}\|$. To impose
sparsity on ${\mathbf{x}^\star}$ and $\mathbf{A}$, we randomly set to zero 95\% of their components. Finally, in \eqref{lasso}, we set $\lambda=1$.

\noindent\textbf{Network setting}. We consider  a fixed, undirected network composed of 50 agents;   $\mathbf{x}\in\mathbb{R}^{30000}$ is partitioned in 50 block-variables $\mathbf{x}_i\in\mathbb{R}^{600}$, $i\in\{1,\ldots,50\}$, each of them controlled by one  agent. We define the local functions $f_i$ and $g_i$ as described in   Sec. \ref{prob_stat} (cf. Ex. \#2);  each $\mathbf{A}_i$ (resp. $\mathbf{b}_i$) is all zeros but    its $i$th row (resp. component),  which coincides with that of $\mathbf{A}$ (resp. $\mathbf{b}$). This induces the following communication pattern among the agents:  each agent $i$ is connected   only to the agents $j$s owning the   $\mathbf{x}_j$s corresponding to the nonzero column-entries of $\mathbf{A}_i$.

\noindent{\textbf{Algorithms.} We simulated the following algorithms:}

 {$\bullet$ \textbf{DAsyFLEXA}: we used}    the   surrogate functions 
\begin{align}
\label{lassosurr}
&\tilde{f}_i\left(\mathbf{x}_i;\mathbf{x}_{\mathcal{N}_i}^{k-\mathbf{d}^k(i,i)}\right)\\ \notag &
=\left\langle\nabla_{\mathbf{x}_i}f_i\left(\mathbf{x}_{\mathcal{N}_i}^{k-\mathbf{d}^k(i,i)}\right),\mathbf{x}_i-\mathbf{x}_i^k\right\rangle+\frac{\tau_i}{2}\|\mathbf{x}_i-\mathbf{x}_i^k\|_2^2,
\end{align}
 where $\tau_i>0$ is a tunable parameter, which is updated following the same heuristic used in \cite{cannelli2017asynchronous}. The stepsize $\gamma$ is set to $0.9$. Note that, using \eqref{lassosurr}, problem \eqref{best_resp} has a closed-form solution via the renowned soft-thresholding operator.
 
 {$\bullet$ \textbf{PrimalDual} asynchronous algorithm \cite{wu2018decentralized}: }this seems to be the   distributed asynchronous scheme  closest to DAsyFLEXA.  {Note that} there are some important differences between the two algorithms.  First,   the PrimalDual  algorithm \cite{wu2018decentralized}   does not exploit the sparsity pattern of the objective function $V$; every agent instead  controls and updates a local copy of the {\it entire }vector $\mathbf{x}$, which requires employing a consensus mechanism to enforce an agreement  on such local copies. This leads to an unnecessary communication overhead among the agents. Second, no explicit estimate of the gradients of the other agents  is employed; the lack of this knowledge is overcome by introducing additional communication variables, which lead to contribute to increase the communication cost.  
Third,  the PrimalDual  algorithm does not have convergence guarantees in the  nonconvex case. In our simulations we tuned the stepsizes of \cite{wu2018decentralized} by hand  in order to obtain the best performances; specifically we set $\alpha=0.9$, and $\eta_i=1.5$ for $i=1,\ldots,50$ (see \cite{wu2018decentralized} for details on these parameters). 

 {$\bullet$\textbf{AsyBADMM}: this is a block-wise asynchronous ADMM, introduced in \cite{zhu2018block} to solve nonconvex and nonsmooth optimization problems. Since AsyBADMM requires the presence of master and worker nodes in the network, to implement it  on a meshed networks,  we selected uniformly at random 5 nodes of the network as servers  while the others acting as workers. The parameters of the algorithm (see \cite{zhu2018block} for details) are tuned by hand in order to obtain the best performances; specifically we set $\gamma=0.06$, $C=10^{4}$, and $\rho_{ij}=50$, for all $(i,j)$.}}

 {All the algorithms   are initialized from the same randomly chosen point, drawn from $\mathcal{N}(0,1)$.}

\noindent{\textbf{Asynchronous model.}} We simulate  the following asynchronous model. Each agent is activated periodically, every time a local clock triggers. The agents' local clocks have the same frequency but   different phase shift,  which are selected uniformly   at random within  $[5, 50]$. Based upon its activation, each agent: i)   performs its  update  and then broadcasts its gradient vector $\nabla_{\mathbf{x}_i}f_i$ together with its own block-variable $\mathbf{x}_i$ to the agents in $\mathcal{N}_i\backslash\{i\}$; and ii)  modifies the phase shift of its local clock by selecting uniformly at random a new   value. 


Figure \ref{time-lasso} plots   relative error $({V(\mathbf{x}^k)-V^\star})/{V^\star}$ of the different methods versus the number of iterations.  {Figure \ref{comm-lasso} shows the same function  versus the number of message exchanges per agent}; each scalar variable sent from an agent to one of its neighbor is counted as one message exchanged. All the curves are averaged over   10 independent realizations.
\begin{figure}[h]
	\centering\vspace{-0.4cm}
	\includegraphics[scale=0.27]{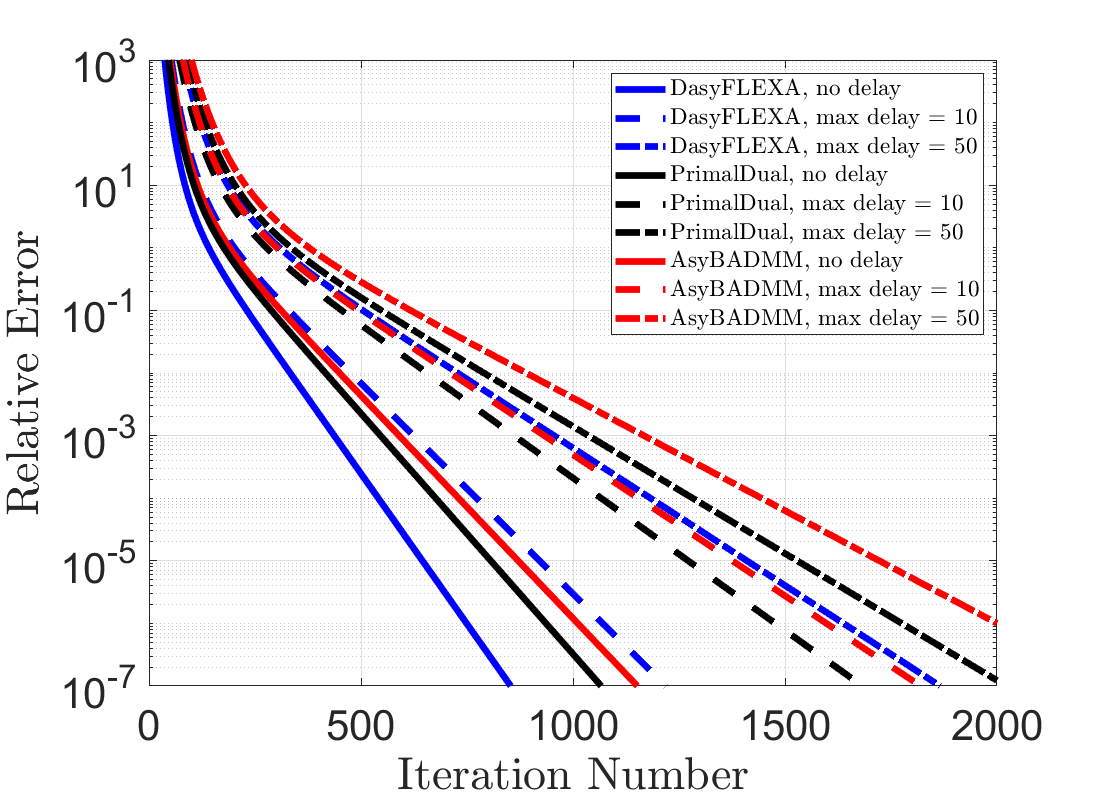}
	\caption{LASSO problem: Relative error vs. \#  of iterations.}
	\label{time-lasso}\vspace{-0.1cm}
\end{figure}

\begin{figure}[h]\vspace{-0.7cm}
	\centering
	\includegraphics[scale=0.27]{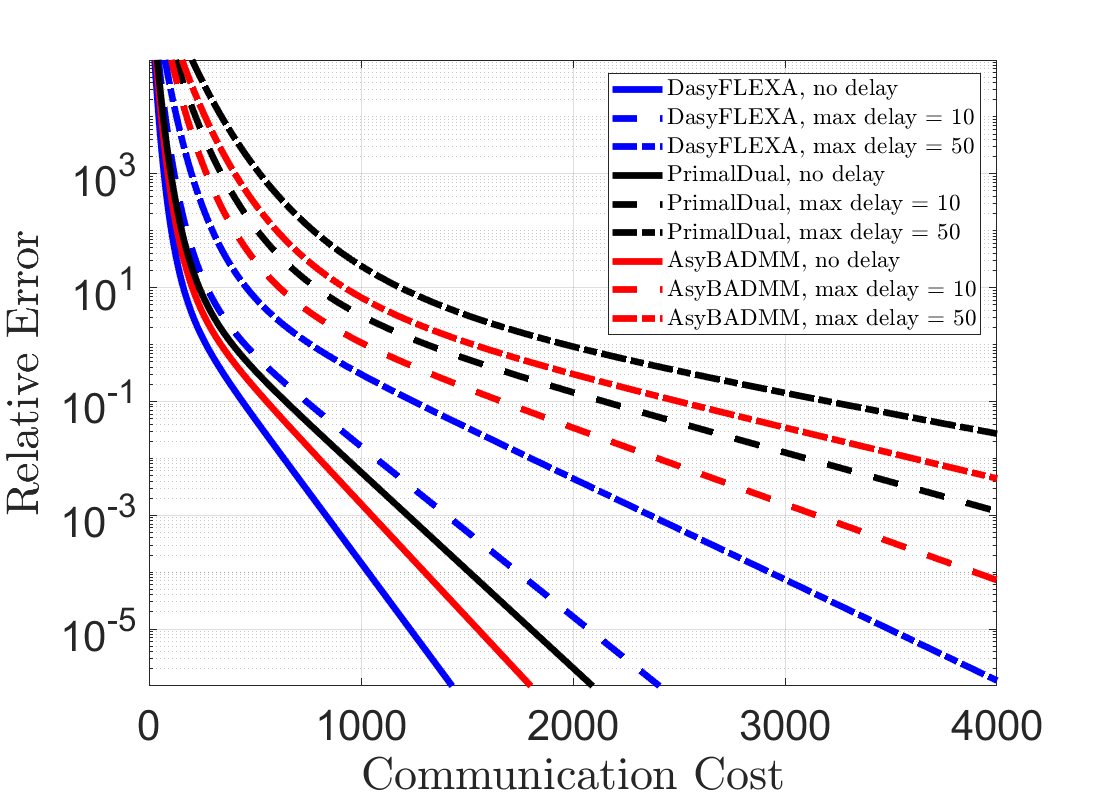}
	\caption{LASSO problem: Relative error vs. \#  of message exchanges.}\vspace{-0.2cm}
	\label{comm-lasso}
\end{figure}

 {DAsyFLEXA  outperforms the PrimalDual scheme \cite{wu2018decentralized} and AsyBADMM \cite{zhu2018block}. Also,  as anticipated,  PrimalDual requires much more communications than DAsyFLEXA.} 
 \vspace{-0.3cm}

\subsection{Distributed Matrix Completion}
In this section we consider the   Distributed Matrix Completion problem \eqref{mc}.
We generate a $2200\times 2200$ matrix $\mathbf{Z}$ with samples drawn from $\mathcal{N}(0,1)$; and we set  $\lambda=\xi=1$ and $r=4$. {Each core of our cluster computer represents a different agent; the columns of $\mathbf{X}$  and $\mathbf{Y}$ are equally partitioned  across the 22 cores, and those of $\mathbf{Y}$ uniformly among the other 11 cores; and all cores access a shared memory where the data are stored.} We sampled uniformly at random 10\% of the entries of $\mathbf{Z}$, and  distributed these samples $z_{mn}$  to the agents owing the corresponding column $\mathbf{x}_m$ of $\mathbf{X}$ or  $\mathbf{y}_n$ of  $\mathbf{Y}$, choosing randomly between the two.

We applied the following instance of  DAsyFLEXA to \eqref{mc}. Consider one of the agents that optimizes some columns of $\mathbf{X}$, say agent $i$. Since each $f_i$ is biconvex in  $\mathbf{X}$ and $\mathbf{Y}$, the following surrogate function satisfies Assumption C:
\begin{align}
\label{surr}
&\tilde{f}_i\left(\{\mathbf{x}_m\}_{m\in X_i};(\mathbf{X},\mathbf{Y})_{\mathcal{N}_i}^{k-\mathbf{d}^k(i,i)}\right)\\ \notag &
=\!\dfrac{1}{2}\sum\limits_{(m,n)\in\Omega_i}\!\!\!\!\left(\mathbf{x}_m^T\mathbf{y}_n^{k-d_{j(n)}^k(i,i)}\!-\!z_{mn}\right)^2
\!\!+\dfrac{\tau_i}{2}\!\!\!\!\sum\limits_{(m,n)\in\Omega_i}\!\!\!\!\! \|\mathbf{x}_m\!-\!\mathbf{x}_m^k\|_2^2;
\end{align}
where $j(n)$ is the index $j\in\mathcal{N}_i$ of the agent that controls   $\mathbf{y}_n$, and $\tau_i>0$ is  updated following the same heuristic used in \cite{cannelli2017asynchronous} (the surrogate function for the agents that update columns of $\mathbf{Y}$ is the same as \eqref{surr}, with the obvious change of variables). Note that \eqref{surr} preserves the block-wise  convexity present in the original function $f_i$, which contrasts  with the common approach in the literature based on the linearization of $f_i$.  Problem \eqref{best_resp} with the surrogate  \eqref{surr} has a closed-form solution. 

 {We compare our algorithm with  the decentralized ADMM version of ARock, as presented in \cite{peng2016arock}. Even if this method has convergence guarantees for convex problems only, its performances on this experiment appeared to be good. For ARock we fixed $\eta^k=0.9$, for all $k$, and $\gamma=10$, which are the values that gave us the best performances in the experiments.}


The rest of the setup   is the same as that described for the LASSO problem.
Figure \ref{time-mc} and Figure \ref{comm-mc} plot $\|\widehat{\mathbf{x}}^k-\mathbf{x}^k\|_{\infty}$ (a valid measure of stationarity), with   $\widehat{\mathbf{x}}_i$ defined as in \eqref{best_resp}, obtained by DAsyFLEXA and the PrimalDual algorithm \cite{wu2018decentralized}  versus  {the CPU time (measured in seconds) and message exchanges per agent.} On our tests, we observed that all the algorithms converged to the same stationary solution.
 The results confirm the behavior observed in the previous section for convex problems: DAsyFLEXA has better performances than PrimalDual, and the difference is mostly significant is terms of communication cost.  {DAsyFLEXA is also more efficient than ARock, which suffers from a similar drawback of PrimalDual for what concerns the number of message exchanges; this is due to the fact that  ARock requires the use of dual variables, which cause a communication overhead.}
\begin{figure}[t]
	\centering\vspace{-.5cm}
	\includegraphics[scale=0.27]{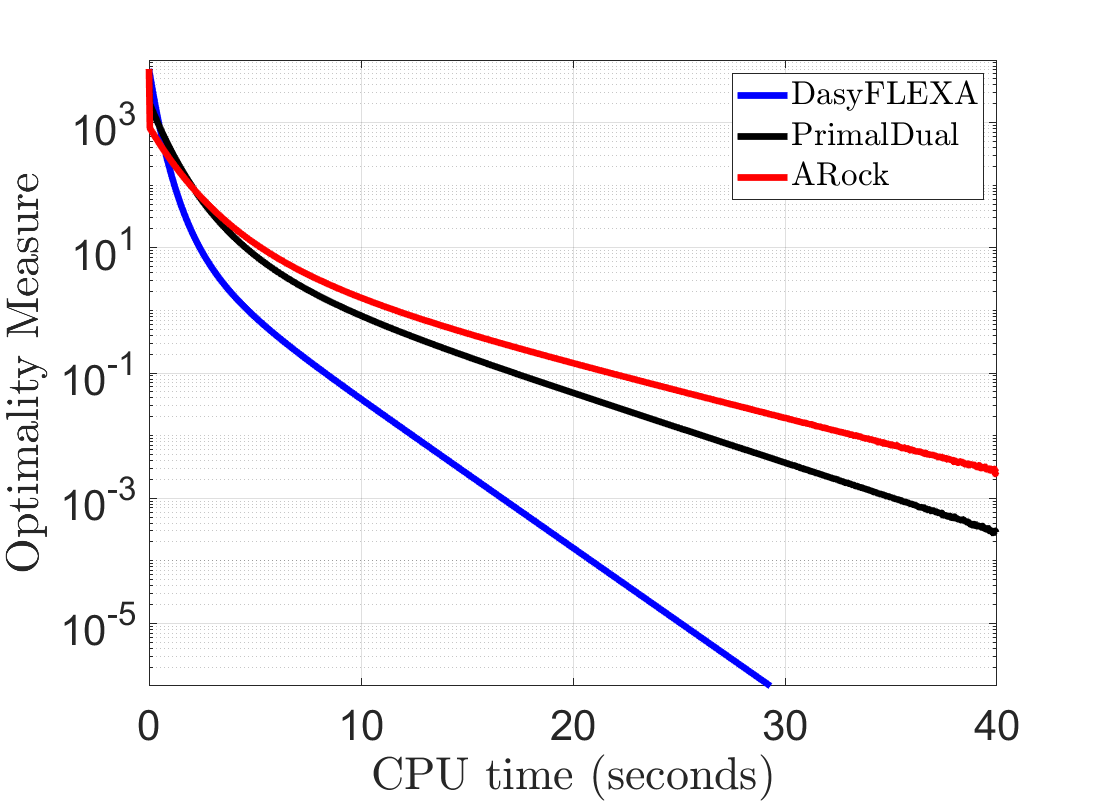}
	\caption{Matrix completion: stationarity distance vs. \# CPU time (in seconds). }
	\label{time-mc}
\end{figure}
\begin{figure}[h]
	\centering\vspace{-0.5cm}
	\includegraphics[scale=0.27]{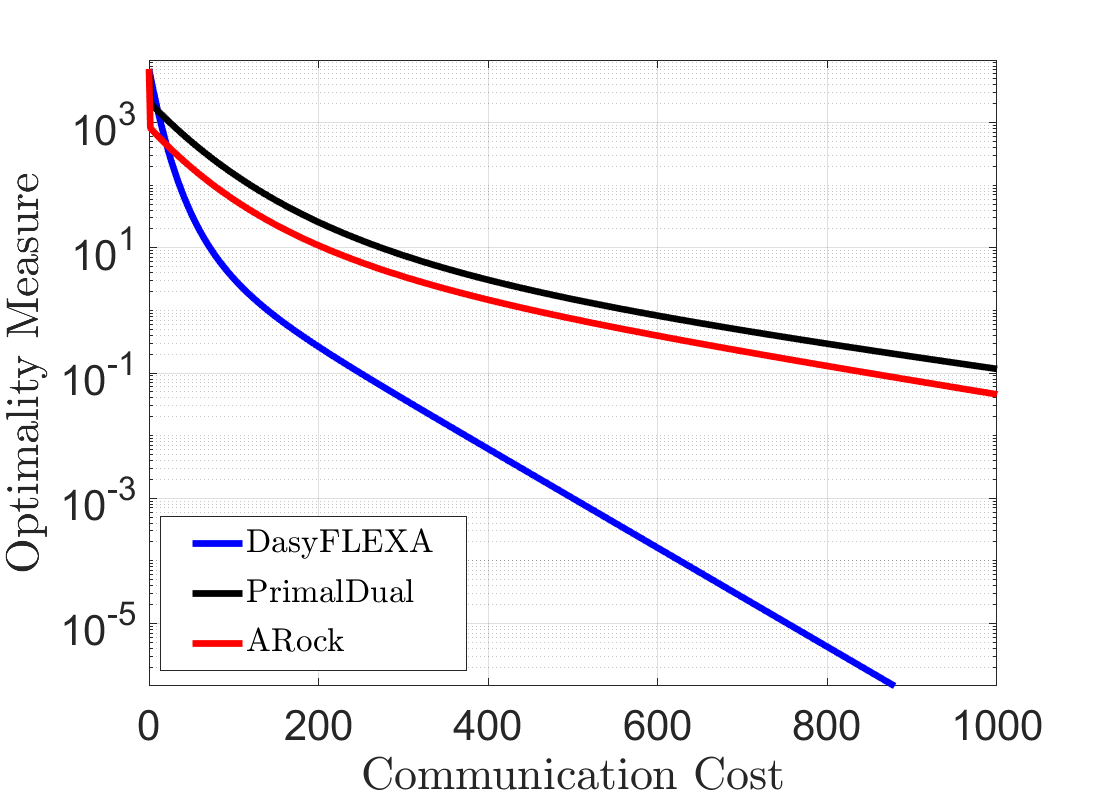}
	\caption{Matrix completion: stationarity distance vs. \# of  message exchanges. }
	\label{comm-mc} \vspace{-0.4cm}
\end{figure}
\vspace{-0.3cm}
\section{Appendix}\vspace{-0.1cm}
In this Appendix we prove   Theorems \ref{compl} and \ref{linear}. 
\vspace{-0.5cm}
\subsection{Notation}\label{notation}
	Vectors $\mathbf{x}_{\mathcal{N}_j}^{k-\mathbf{d}^k(i,j)}$ have different length. It is convenient to replace them with equal-length vectors retaining of course the same information. This is done   introducing   the following (column) vectors $\mathbf{x}^k(i,j)\triangleq (\mathbf{x}^k_l(i,j))_{l=1}^N\in\mathcal{X}$, defined as: 

\begin{subequations}\label{eq:equal-lenght_vec}
		\begin{equation}	
		[\mathbf{x}^k_l(i,j)]_{l\in\mathcal{N}_j}\triangleq\mathbf{x}_{\mathcal{N}_j}^{k-\mathbf{d}^k(i,j)}, \end{equation}
			\begin{equation}	 \mathbf{x}_l^k(i,j)=\mathbf{x}^k_l,\quad l\notin\mathcal{N}_j.\end{equation}
	\end{subequations}  
	In words, the blocks of  $\mathbf{x}^k(i,j)$  indexed by $\mathcal{N}_j$   coincide  with $\mathbf{x}_{\mathcal{N}_j}^{k-\mathbf{d}^k(i,j)}$   whereas the other block-components, irrelevant to the proofs,  are conveniently set   to their most up-to-date values. We will use the shorthand $\mathbf{x}_{\mathcal{N}_j}^k(i,j)\triangleq
		[\mathbf{x}^k_l(i,j)]_{l\in\mathcal{N}_j}$. 

	Since at each iteration $k\geq0$  only one block of   $\mathbf{x}^k$ is updated, and because of Assumption D2, 
	it is not difficult to check that the delayed vector $\mathbf{x}^k(i,j)$ can be written as\vspace{-0.1cm}
\begin{equation}
\mathbf{x}^k(i,j)=\mathbf{x}^k+\sum\limits_{l\in \mathcal{K}^k(i,j)}(\mathbf{x}^l-\mathbf{x}^{l+1}),
\label{x_tilde}\vspace{-0.1cm}
\end{equation}
where $\mathcal{K}^k(i,j)$ is a subset of $\{k-D,\ldots,k-1\}$ whose elements depend on which block variables have been updated in the window $[\max\{0,k-D\},\max\{0,k-1\}]$. Recall that it is assumed $\mathbf{x}^t=\mathbf{x}^0$, for $t<0$.

	Finally,  { notice that the notation $\widehat{\mathbf{x}}_i^k$ for the best-response map \eqref{best_resp} is a shorthand for the formal expression $\widehat{\mathbf{x}}_i(\tilde{\mathbf{x}}^k(i))$, where $\tilde{\mathbf{x}}^k(i)\triangleq\left[\mathbf{x}_{\mathcal{N}_j}^{k-\mathbf{d}^k(i,j)}\right]_{j\in\mathcal{N}_i}$. Similarly, $\widehat{\bar{\mathbf{x}}}_i^k$ (resp. $\widehat{\bar{\mathbf{x}}}^k$) in \eqref{def bar x} is a shorthand for $\widehat{\mathbf{x}}_i(\bar{\mathbf{x}}^k(i))$ (resp. $\widehat{\mathbf{x}}(\bar{\mathbf{x}}^k)$), where $\bar{\mathbf{x}}^k(i)\triangleq\left[\mathbf{x}^k_{\mathcal{N}_j}\right]_{j\in\mathcal{N}_i}$ (resp. $\bar{\mathbf{x}}^k\triangleq\left[\bar{\mathbf{x}}^k(i)\right]_{i\in\mathcal{N}}$). We also define the following shorthands}: \vspace{-0.1cm}
\begin{equation}\label{subs2}
		\Delta\widehat{\mathbf{x}}^k\triangleq\left[\Delta\widehat{\mathbf{x}}_i^k\right]_{i\in\mathcal{N}},\quad \Delta\widehat{\mathbf{x}}_i^k\triangleq\widehat{\mathbf{x}}_i^k-\mathbf{x}_i^k.
	\end{equation}Table \ref{tabella} summarizes the main notation used in the paper.
\noindent\textbf{On the  constant $L$.} The proofs rely on some Lipschitz properties of
$\nabla_{\mathbf{x}_j}f_i$'s. To provide a unified proof under either   A2(a) or A2(b), we introduce  a constant $L>0$ whose value depends on whether A2(a) or A2(b) hold. Specifically: \\
\noindent$\bullet$ {\bf A2(a) holds}:  the gradients $\nabla_{\mathbf{x}_j}f_i$'s are not globally Lipschitz on the sets $\mathcal{X}_{\mathcal{N}_i}$'s; our approach to study convergence is to ensure that they are Lipschitz continuous on suitably defined sets containing the sequences  
generated by Algorithm~\ref{algorithm1}.
We define these sets as follows. Define first the set $\texttt{Cube}\triangleq\left\{\mathbf{w}\in\mathcal{X}:\|\mathbf{w}\|_{\infty}\leq U\right\}$, where $U$ is positive constant that ensures $\mathcal{L}^0\subseteq\texttt{Cube}$ (note that $U<+\infty$ because $\mathcal{L}^0$ is bounded). Then, we 
define a proper widening $\bar{\mathcal{L}}^0$ of $\mathcal{L}^0$: $\bar{\mathcal{L}}^0\triangleq\left(\mathcal{L}^0+\psi\mathcal{B}\right)\cap\mathcal{X}$, where $\mathcal{B}$ is the unitary ball centered in the origin, and $\psi>0$ is a finite positive constant defined as
\begin{equation}
\psi\triangleq\underset{i\in\mathcal{N}}{\max}\;\underset{\substack{\tilde{\mathbf{w}}(i)\triangleq[\mathbf{w}_{\mathcal{N}_j}(j)]_{j\in\mathcal{N}_i}\\\mathbf{w}(j)\in\texttt{Cube}}}{\max}\;\|\widehat{\mathbf{x}}_i\left(\tilde{\mathbf{w}}(i)\right)-\mathbf{w}_i(i)\|_2\label{psi}.
\end{equation}Note that $\bar{\mathcal{L}}^0$ is compact, because $\mathcal{L}^0$ is bounded and $\psi<+\infty$ [given that $\texttt{Cube}$ is bounded and $\widehat{\mathbf{x}}(\cdot)$ is continuous, due to  \eqref{best_resp}, A2, A3, and C3]. Consider now any vector $\mathbf{x}\in\bar{\mathcal{L}}^0$. A2(a) and compactness of $\bar{\mathcal{L}}^0$ imply that the gradients $\nabla_{\mathbf{x}_j}f_i$'s are globally Lipschitz over the sets containing the subvectors $\mathbf{x}_{\mathcal{N}_i}$'s, with $L$ being the maximum value of the Lipschitz constant of all the gradients over these sets.\\
\noindent$\bullet$ {\bf A2(b) holds:} In this case, $L$ is simply the global Lipschitz constant $\nabla_{\mathbf{x}_i} f_i$ over the whole space.	
\begin{table}[h]	
\centering\small
\begin{tabular}{|>{\centering\arraybackslash}m{1.2in}| >{\centering\arraybackslash}m{2in}|} 
\hline
\textbf{Symbol} & \textbf{Definition}\\
\hline\hline
$V(\mathbf{x})$, cf. \eqref{problem} & $F(\mathbf{x})+G(\mathbf{x})$ \\ 
\hline
$F(\mathbf{x})$, cf. \eqref{problem}  & $\sum\limits_{i=1}^Nf_i\left(\mathbf{x}_{\mathcal{N}_i}\right)$\\
\hline
$G(\mathbf{x})$, cf. \eqref{problem}  & $\sum\limits_{i=1}^Ng_i\left(\mathbf{x}_i\right)$\\
\hline
$\mathbf{x}$, cf. \eqref{problem}  & Optimization variable\\
\hline
$\mathbf{x}_i$, cf. \eqref{problem}  & Block-variable of agent $i$\\
\hline
$\mathbf{x}_{\mathcal{N}_i}$, cf. \eqref{problem}  & Block-variables of agent $i$'s set of neighbors: $[\mathbf{x}_j]_{j\in\mathcal{N}_i}$\\
\hline
$\mathbf{x}_{\mathcal{N}_j}^{k-\mathbf{d}^k(i,j)}$ & Agent $i$'s local copy of agent $j$'s vector $\mathbf{x}_{\mathcal{N}_j}^k$, possibly delayed\\
\hline
$\mathbf{x}^k(i,j)$, cf.  \eqref{eq:equal-lenght_vec}&Same as $\mathbf{x}_{\mathcal{N}_j}^{k-\mathbf{d}^k(i,j)}$, with the addition of slack elements to fix dimensionality\\
\hline
$\widehat{\mathbf{x}}_i^k/\widehat{\mathbf{x}}_i(\tilde{\mathbf{x}}^k(i))$, cf. \eqref{best_resp} & Solution of subproblem \eqref{best_resp}\\
\hline
$\tilde{\mathbf{x}}^k(i)$ & Agent $i$'s local copies of his neighbors vectors $[\mathbf{x}_{\mathcal{N}_j}^k]_{j\in\mathcal{N}_i}$, possibly delayed: $\left[\mathbf{x}_{\mathcal{N}_j}^{k-\mathbf{d}^k(i,j)}\right]_{j\in\mathcal{N}_i}$\\
\hline
$\bar{\mathbf{x}}^k$& Collection of all the $\tilde{\mathbf{x}}^k(i)$'s: $\left[\bar{\mathbf{x}}^k(i)\right]_{i\in\mathcal{N}}$\\
\hline
$\widehat{\bar{\mathbf{x}}}_i^k/\widehat{\mathbf{x}}_i(\bar{\mathbf{x}}^k(i))$, cf. \eqref{def bar x} & Solution of subproblem \eqref{best_resp} wherein all the delays are set to 0\\
\hline
$\bar{\mathbf{x}}^k(i)$& Same structure of $\tilde{\mathbf{x}}^k(i)$ wherein all the delays are set to 0\\
\hline
$\bar{\mathbf{x}}^k$& Collection of all the $\bar{\mathbf{x}}^k(i)$'s: $\left[\bar{\mathbf{x}}^k(i)\right]_{i\in\mathcal{N}}$\\
\hline
\end{tabular}\caption{Table of notation}\label{tabella}
\end{table}\vspace{-.1cm}
\begin{remark}\label{remark_on_L}
To make sense of the complicated definition of
$L$ under A2(a), we anticipate how this constant will be used. Our proof leverages the decent lemma to majorize $V(\mathbf{x}^{k+1})$. To do so, each $\nabla_{\bx_j}f_i$ needs to be globally Lipschitz on a convex set containing $\bx^k$ and $\bx^{k+1}$. This is what the convex set $\bar{\mathcal{L}}^0$  is meant for: $\bx^k$ and $\bx^{k+1}$ belong to $\bar{\mathcal{L}}^0$ and thus $\nabla_{\bx_j}f_i$ is   $L$-Lipschitz continuous.
\end{remark}
\vspace{-0.4cm}
\subsection{Preliminaries}\label{technical}
	 We summarize next some properties of the   map $\widehat{\mathbf{x}}_i^k$ in \eqref{best_resp}.
	\begin{proposition} \label{Prop_best_response}Given Problem \eqref{problem} under Assumption A, let $\{\mathbf{x}^k\}$ be the sequence generated by DAsyFLEXA, under Assumptions B and C. Suppose also that
 $\mathbf{x}^k \in \bar{\mathcal{L}}^0 $ for all $k$. There hold:
		\begin{itemize}			
			\item [(a)][Optimality] For any $i\in\mathcal{N}$ and $k\geq0$,
			\begin{align}
			\nonumber
			&
			 \sum\limits_{j\in\mathcal{N}_i}\left\langle\nabla_{\mathbf{x}_i}f_j(\mathbf{x}^k_{\mathcal{N}_j}(i,j)),\Delta\widehat{\mathbf{x}}^k_i\right\rangle\\&+g_i(\widehat{\mathbf{x}}_i^k)-g_i(\mathbf{x}^k_i)\leq-\tau\|\Delta\widehat{\mathbf{x}}_i^k\|_2^2;\label{comments}
			\end{align}
			\item [(b)][Lipschitz continuity] For any $i\in\mathcal{N}$ and $k,h\geq0$,\vspace{-0.2cm}
			\begin{align}
			&\nonumber\|\widehat{\mathbf{x}}_i^k-\widehat{\mathbf{x}}_i^h\|_2\leq \label{lipsch}\frac{L_m}{\tau}\|\mathbf{x}^k(i,i)-\mathbf{x}^h(i,i)\|_2\\
			&\qquad +\frac{L}{\tau}\sum\limits_{j\in\mathcal{N}_i\backslash\{i\}}\|\mathbf{x}^k(i,j)-\mathbf{x}^k(i,j)\|_2,
			\end{align}
			where $L_m\triangleq\underset{i\in\mathcal{N}}{\max}\,L_i$;
			\item [(c)][Fixed-points]   $\widehat{\mathbf{x}}(\bar{\mathbf{x}}^k)=\mathbf{x}^k$ if and only if $\mathbf{x}^k$ is a stationary solutions of Problem \eqref{problem} (recall the definition of $\bar{\mathbf{x}}^k$  {in Table \ref{tabella}});
			\item [(d)][Error bound]\vspace{-0.1cm}\label{pro:preliminary}		
				For any $k\geq0$,
				\begin{align}
				&\nonumber\|\mathbf{x}^k-\text{\texttt{prox}}_{G}\left(\mathbf{x}^k-\nabla F(\mathbf{x}^k)\right)\|_2\\
				&\label{lemma3}\leq(1+L+NL_m)\|\widehat{\mathbf{x}}(\bar{\mathbf{x}}^k)-\mathbf{x}^k\|_2.
				\end{align}
		\end{itemize}
	
	\end{proposition}

	\begin{proof} We prove only (d); the proof of
		(a)-(c) follows similar steps of that in  \cite[Proposition 8]{FLEXA}, and thus is omitted. 		 Invoking the optimality of  $\widehat{\mathbf{x}}(\bar{\mathbf{x}}^k)$, we have
		\begin{align*}&\Bigg\langle\nabla\tilde{f}_i\left(\widehat{\mathbf{x}}_i(\bar{\mathbf{x}}^k(i));\mathbf{x}^k_{\mathcal{N}_i}\right)+\sum\limits_{j\in\mathcal{N}_i\backslash\{i\}}\nabla_{\mathbf{x}_i}f_j\left(\mathbf{x}^k_{\mathcal{N}_j}\right),\\&\widehat{\mathbf{x}}_i(\bar{\mathbf{x}}^k(i))-\mathbf{z}_i\Bigg\rangle+g_i(\widehat{\mathbf{x}}_i(\bar{\mathbf{x}}^k(i)))-g_i(\mathbf{z}_i)\leq0,\end{align*}
		for all $\mathbf{z}\in\mathcal{X}$ and   $i\in\mathcal{N}$. Setting $\check{\mathbf{x}}^k\triangleq\texttt{prox}_G\left(\mathbf{x}^k-\nabla F(\mathbf{x}^k)\right)$, and invoking the variational characterization of the proximal operator, we have  		\begin{equation*}\left\langle\nabla_{\mathbf{x}}F(\mathbf{x}^k)+\check{\mathbf{x}}^k-\mathbf{x}^k,\check{\mathbf{x}}^k-\mathbf{w}\right\rangle+G(\check{\mathbf{x}}^k)-G(\mathbf{w})\leq0,\end{equation*}
	for all	$\mathbf{w}\in\mathcal{X}$.
		Summing the two  inequalities  above, with $\mathbf{z}=\check{\mathbf{x}}^k$, $\mathbf{w}=\widehat{\mathbf{x}}(\bar{\mathbf{x}}^k)$, and   using C1 and C2, yields \begin{align}\nonumber&\tau\|\widehat{\mathbf{x}}(\bar{\mathbf{x}}^k)-\mathbf{x}^k\|_2^2+\|\check{\mathbf{x}}^k-\mathbf{x}^k\|_2^2\\
		\nonumber&\leq \|\check{\mathbf{x}}^k-\mathbf{x}^k\|_2\|\widehat{\mathbf{x}}(\bar{\mathbf{x}}^k(i))-\mathbf{x}^k\|_2+\sum\limits_{i=1}^N\Big\langle\nabla\tilde{f}_i(\widehat{\mathbf{x}}_i(\bar{\mathbf{x}}^k(i));\mathbf{x}^k_{\mathcal{N}_i})\\
		\nonumber&\quad -\nabla\tilde{f}_i(\mathbf{x}_i^k;\mathbf{x}^k_{\mathcal{N}_i}),\check{\mathbf{x}}_i^k-\mathbf{x}_i^k\Big\rangle\\
		\nonumber&\stackrel{A2, C2-C3}{\leq} \|\mathbf{x}^k-\check{\mathbf{x}}^k\|_2\left((1+L+NL_m)\|\widehat{\mathbf{x}}(\bar{\mathbf{x}}^k)-\mathbf{x}^k\|_2\right).\end{align}
	\end{proof}\vspace{-0.3cm}
\subsection{Proof of Theorem \ref{compl}}\label{proof_general}
 The proof is organized in the following steps:

\noindent \textbf{Step 1--Lyapunov function \& its descent:} We define an appropriate Lyapunov  function $\tilde{V}$ and   prove that it is monotonically nonincreasing along the iterations. This also proves Theorem \ref{compl}(c);

\noindent  \textbf{Step 2--Vanishing} $\mathbf{x}$\textbf{-stationarity:} Building on the descent properties of the Lyapunov function, we prove   $\lim_{k\rightarrow+\infty}\|\widehat{\mathbf{x}}(\bar{\mathbf{x}}^k)-\mathbf{x}^k\|_2=0$ [Theorem \ref{compl}(a)];

\noindent\textbf{Step 3--Convergence rate:} We prove the sublinear convergence rate of $\{M_V(\mathbf{x}^k)\}$ as stated in Theorem \ref{compl}(c). 

\noindent The above steps are proved under Assumptions A, C, and D.

\subsubsection{\textbf{Step 1--Lyapunov function \& its descent}}\label{step1}
Introduce the following Lyapunov-like  function:
\begin{align}
\nonumber&\tilde{V}(\mathbf{x}^k,\dots,\mathbf{x}^{k-D})
&\triangleq V(\mathbf{x}^k)+\frac{DL\rho^2}{2}\Bigg(\sum\limits_{l=k-D}^{k-1}\big(l-(k-1)
\\&&+D\big)\|\mathbf{x}^{l+1}-\mathbf{x}^l\|_2^2\Bigg),\label{lyap}
\end{align}
where $L$ is defined in Sec.\ref{notation}.
Note that  $$\tilde{V}^\star\triangleq\underset{[\mathbf{y}^i\in\mathcal{X}]_{i=1}^{D+1}}{\min}\,\tilde{V}(\mathbf{y}^1,\ldots,\mathbf{y}^{D+1})=\underset{\mathbf{x}\in\mathcal{X}}{\min}\,V(\mathbf{x}).$$

The following lemma establishes the descent properties of $\tilde{V}$
and also proves Theorem \ref{compl}(c). 	
\begin{lemma}\label{descent}
	Given $\tilde{V}$ defined in \eqref{lyap}, the following hold:\\\vspace{-0.1cm}
	(a) For any $k\geq0$:
	\begin{align}\label{eq:final0_multi}&\tilde{V}(\mathbf{x}^{k+1}\ldots,\mathbf{x}^{k+1-D})\\\nonumber&\leq \tilde{V}(\mathbf{x}^k,\ldots,\mathbf{x}^{k-D})-\gamma\left(\tau-\gamma\frac{L\left(2+D^2\rho^2\right)}{2}\right)
	\|\Delta\widehat{\mathbf{x}}_{i^k}^k\|_2^2.\end{align}
	(b) If, in particular,  A2(a) is satisfied: $\mathbf{x}^k\in \mathcal{L}^0, \,\text{for all}\, k\geq 0.$
\end{lemma}
\begin{proof}
	We prove the two statements by induction.
	For $k=0$, \vspace{-0.2cm}
	\begin{align}
	\hspace{-0.2cm}	\nonumber
		 V(\mathbf{x}^1)& = 
		  \sum\limits_{i=1}^Nf_i(\mathbf{x}^1_{\mathcal{N}_i}) + g_{i^0}(\mathbf{x}_{i^0}^1) + \sum\limits_{i \neq i^0} g_i(\mathbf{x}_i^1) \\\nonumber
		&\stackrel{\text{\eqref{update}}}{=} \sum\limits_{i=1}^Nf_i(\mathbf{x}^1_{\mathcal{N}_i}) + g_{i^0}(\mathbf{x}_{i^0}^1) + \sum\limits_{i \neq i^0} g_i(\mathbf{x}_i^0)   \\
		\nonumber
		&\stackrel{\text{(a)}}{\leq} \sum\limits_{i=1}^Nf_i(\mathbf{x}^0_{\mathcal{N}_i})  +\gamma \sum\limits_{j\in\mathcal{N}_{i^0}}\bigg\langle\nabla_{\mathbf{x}_{i^0}} f_j (\mathbf{x}^0_{\mathcal{N}_j})\\
		\nonumber&\quad + \nabla_{\mathbf{x}_{i^0}}
		f_j\left(\mathbf{x}^0_{\mathcal{N}_j}(i^0,j)\right)- \nabla_{\mathbf{x}_{i^0}}
		f_j\left(\mathbf{x}^0_{\mathcal{N}_j}(i^0,j)\right),\Delta\widehat{\mathbf{x}}_{i^0}^0\bigg\rangle \\
		\nonumber&\quad+ \frac{\gamma^2L}{2}\|\Delta\widehat{\mathbf{x}}_{i^0}^0\|_2^2+g_{i^0}(\mathbf{x}_{i^0}^1)+\sum\limits_{i \neq i^0} g_i(\mathbf{x}_i^0)\\
		\nonumber
		&
		\stackrel{\text{A3}}{\leq} \sum\limits_{i=1}^Nf_i(\mathbf{x}^0_{\mathcal{N}_i}) + \gamma\left\langle\sum\limits_{j\in\mathcal{N}_{i^0}}\hspace{-.15cm}\nabla_{\mathbf{x}_{i^0}}
		f_j\left(\mathbf{x}^0_{\mathcal{N}_j}(i^0,j)\right),\Delta\widehat{\mathbf{x}}_{i^0}^0\hspace{-.1cm}\right\rangle\\
		\nonumber&\quad+\gamma\Big\langle\sum\limits_{j\in\mathcal{N}_{i^0}}\Big(\nabla_{\mathbf{x}_{i^0}}
		f_j(\mathbf{x}^0_{\mathcal{N}_j})\\
		\nonumber&\quad-\nabla_{\mathbf{x}_{i^0}}f_j\left(\mathbf{x}^0_{\mathcal{N}_j}(i^0,j)\right)\Big),\Delta\widehat{\mathbf{x}}_{i^0}^0\Big\rangle+\frac{\gamma^2L}{2}\|\Delta\widehat{\mathbf{x}}_{i^0}^0\|_2^2\\
		\nonumber&\quad+ \sum\limits_{i=1}^N g_i(\mathbf{x}_i^0)+\gamma g_{i^0}(\widehat{\mathbf{x}}_{i^0}^0)-\gamma g_{i^0}(\mathbf{x}^0_{i^0})\\
		\nonumber &
		\stackrel{\eqref{comments}, A2}{\leq} V(\mathbf{x}^0) -\gamma\left(\tau-\frac{\gamma L}{2}\right)\|\Delta\widehat{\mathbf{x}}_{i^0}^0\|_2^2\\
		\nonumber &\quad+\gamma L
		\|\Delta\widehat{\mathbf{x}}_{i^0}^0\|_2\sum\limits_{j\in\mathcal{N}_{i^0}}\|\mathbf{x}^0-\mathbf{x}^0(i^0,j)\|_2\\
		&\nonumber\stackrel{\text{(b)}}{\leq} V(\mathbf{x}^0)-\gamma\left(\tau-\gamma L\right)\|\Delta\widehat{\mathbf{x}}_{i^0}^0\|_2^2\\
		&\quad+\frac{L\rho}{2}
		\sum\limits_{j\in\mathcal{N}_{i^0}}\underbrace{\|\mathbf{x}^0-\mathbf{x}^0(i^0,j)\|_2^2}_{\texttt{term I}},
		\label{prelyap_multi_02}
 		\end{align}
		where (a) follows from  the descent lemma and the definition of $L$; and in (b) we used  Young's inequality. Note that in (a) we used the fact that  $\mathbf{x}^0$ and $\mathbf{x}^1$ belong to $\bar{\mathcal{L}}^0$ (cf. Remark \ref{remark_on_L}).

We now bound \texttt{term I} in \eqref{prelyap_multi_02}. It is convenient to study the more general term $\|\mathbf{x}^k-\mathbf{x}^k(i^k,j)\|_2^2$, $j\in\mathcal{N}_{i^k}$. There holds:\vspace{-0.2cm}
\begin{align}
&\nonumber\|\mathbf{x}^k-\mathbf{x}^k(i^k,j)\|_2^2
\stackrel{\eqref{x_tilde}}{\leq}\,\left(\sum\limits_{l=k-D}^{k-1}\|\mathbf{x}^{l+1}-\mathbf{x}^l\|_2\right)^2
\\
&\nonumber
\leq D\sum\limits_{l=k-D}^{k-1}\|\mathbf{x}^{l+1}-\mathbf{x}^l\|_2^2\\
\nonumber &=D\Bigg(\sum\limits_{l=k-D}^{k-1}\left(l-(k-1)+D\right)\|\mathbf{x}^{l+1}-\mathbf{x}^l\|_2^2\\
 &\quad-\sum\limits_{l=k+1-D}^{k}(l-k+D)\|\mathbf{x}^{l+1}-\mathbf{x}^l\|_2^2\Bigg)\label{eq:newbound1_multi}\hspace{-.1cm}+
D^2\|\mathbf{x}^{k+1}-\mathbf{x}^k\|^2_2.
\end{align}
Combining  \eqref{prelyap_multi_02} and \eqref{eq:newbound1_multi} one can check that statements (a) and (b) of the lemma hold at $k=0$, that is, $\tilde{V}\left(\mathbf{x}^1,\ldots,\mathbf{x}^0\right)\leq\tilde{V}\left(\mathbf{x}^0,\ldots,\mathbf{x}^0\right)$, and $V(\mathbf{x}^1)\leq\tilde{V}\left(\mathbf{x}^1,\ldots,\mathbf{x}^0\right)\leq\tilde{V}\left(\mathbf{x}^0,\ldots,\mathbf{x}^0\right)=V(\mathbf{x}^0)$, respectively.\newline
	Assume now that the two statements hold at iteration $k$. It is easy to check that
the analogous of \eqref{prelyap_multi_02}
also holds at iteration $k+1$ with the term $\sum\limits_{j\in\mathcal{N}_{i^k}}\|\mathbf{x}^k-\mathbf{x}^k(i^k,j)\|_2$ in the analogous of \texttt{term I} at iteration $k$, majorized using  \eqref{eq:newbound1_multi}. Combining \eqref{prelyap_multi_02} at $k+1$ with \eqref{eq:newbound1_multi} one can check that statement (a) of the lemma holds at $k+1$.	We also get: $V(\mathbf{x}^{k+1})\leq\tilde{V}\left(\mathbf{x}^{k+1},\ldots,\mathbf{x}^{k+1-D}\right)\stackrel{\eqref{prelyap_multi_02}}{\leq}\tilde{V}\left(\mathbf{x}^k,\ldots,\mathbf{x}^{k-D}\right)\leq\tilde{V}\left(\mathbf{x}^0,\ldots,\mathbf{x}^0\right)=V(\mathbf{x}^0)$, which proves statement (b) of the lemma at $k+1$. This completes the proof.\vspace{-0.2cm}
	\end{proof}

\subsubsection{\textbf{Step 2 -- Vanishing} $\mathbf{x}$\textbf{-stationarity}}\label{step2}	
It follows from A4 and Lemma \ref{descent}  that, if $\gamma<\frac{2\tau}{L\left(2+\rho^2D^2\right)}$, $\{\tilde{V}(\mathbf{x}^{k-D},\ldots,\mathbf{x}^{k})\}$  and thus $\{V(\mathbf{x}^k)\}$ converge. Therefore,
\begin{equation}
\lim\limits_{k\rightarrow+\infty}\|\Delta\widehat{\mathbf{x}}_{i^k}^k\|_2=0.
\label{limit}
\end{equation}

The next lemma extends the vanishing properties of a single block $\Delta\widehat{\mathbf{x}}^k_{i^k}$ to the entire vector $\Delta\widehat{\mathbf{x}}^k$.
\begin{lemma}\label{tworesults} For any $i\in\mathcal{N}, k\geq0$, and $h,t\in[k,k+B-1]$, there hold:

	 \vspace{-0.6cm}
		\begin{equation}
	 \|\widehat{\mathbf{x}}_i(\tilde{\mathbf{x}}^t(i))-\widehat{\mathbf{x}}_i(\tilde{\mathbf{x}}^h(i))\|_2^2\leq C_2\sum\limits_{l=k-D}^{k+B-2}\|\Delta\widehat{\mathbf{x}}^l_{i^l}\|_2^2,\label{lip}\vspace{-0.2cm}
		\end{equation}	
	 \vspace{-0.cm} 	\begin{equation}
		 \|\Delta\widehat{\mathbf{x}}^h\|_2^2\leq2\Big(NC_2+1\Big)\sum\limits_{l=k-D}^{k+B-1}\|\Delta\widehat{\mathbf{x}}^l_{i^l}\|_2^2.\label{fourth_bis_a}\vspace{-0.3cm}
		\end{equation}
		 with \begin{equation*}
		C_2\triangleq\frac{3\gamma^2(B+2D-N+1)\rho\left(L_m^2+(\rho-1)L^2\right)}{\tau^2}.
		\end{equation*}
\end{lemma}
\begin{proof}
	See Section \ref{suppl}.\vspace{-0.1cm}
\end{proof}

Using \eqref{fourth_bis_a} and \eqref{limit} yields\begin{equation}\lim\limits_{k\rightarrow+\infty}\|\Delta\widehat{\mathbf{x}}^k\|_2=0.\label{new}\end{equation} Furthermore, invoking   \eqref{limit}, \eqref{lip}, and \eqref{new} together with  $\|\widehat{\mathbf{x}}(\bar{\mathbf{x}}^k)-\mathbf{x}^k\|_2\leq\|\Delta\widehat{\mathbf{x}}^k\|_2+\|\widehat{\mathbf{x}}(\bar{\mathbf{x}}^k)-\widehat{\mathbf{x}}^k\|_2$, leads to
\begin{equation}
\lim\limits_{k\rightarrow+\infty}\|\widehat{\mathbf{x}}(\bar{\mathbf{x}}^k)-\mathbf{x}^k\|_2=0,
\label{limit2}
\end{equation}
which, together  with  Proposition \ref{Prop_best_response}(c), proves Theorem \ref{compl}(a).
\subsubsection{\textbf{Step 3 -- Convergence rate}}\label{step3}
We use the Lyapunov function $\tilde{V}$ to study the vanishing rate of $\{M_V(\mathbf{x}^k)\}$. Due to \eqref{limit2} and the definition of $M_V$, we know that $M_V$ is converging to 0. Therefore $T_{\epsilon}$ is finite. Using   $M_V(\mathbf{x}^k)>\epsilon$, for all $k\in\{0,\ldots,T_\epsilon-1\}$, we have
\begin{align}
\nonumber T_\epsilon\epsilon &\leq\sum\limits_{k=0}^{T_\epsilon-1}M_V(\mathbf{x}^k)
 \leq2\sum\limits_{k=0}^{T_\epsilon-1}\left(\|\Delta\widehat{\mathbf{x}}^k\|_2^2+\|\widehat{\mathbf{x}}(\bar{\mathbf{x}}^k)-\widehat{\mathbf{x}}^k\|_2^2\right)\\
\nonumber&\stackrel{\eqref{lipsch},\eqref{fourth_bis_a}}{\leq}2\sum\limits_{k=0}^{T_\epsilon-1}\Bigg(2\left(NC_2+1\right)\sum\limits_{l=k-D}^{k+B-1}\|\Delta\widehat{\mathbf{x}}_{i^l}^l\|_2^2\\
\nonumber&\quad+\sum\limits_{i=1}^N\Bigg(\frac{L_m^2\rho}{\tau^2}\|\mathbf{x}^k(i,i)-\mathbf{x}^k\|_2^2\\
\nonumber&\quad+\frac{L^2\rho}{\tau^2}\sum\limits_{j\in\mathcal{N}_i\backslash\{i\}}\|\mathbf{x}^k(i,j)-\mathbf{x}^k\|_2^2\Bigg)\Bigg)\\
\nonumber&\stackrel{\eqref{x_tilde}}{\leq}2\left(2\left(NC_2+1\right)+\frac{DC_2}{3(B+2D-N+1)}\right)\\
\nonumber& \quad \cdot \sum\limits_{k=0}^{T_\epsilon-1}\sum\limits_{l=k-D}^{k+B-1}\|\Delta\widehat{\mathbf{x}}^l_{i^l}\|_2^2\\
\nonumber&\stackrel{(a)}{\leq}C_3\sum\limits_{k=0}^{T_\epsilon-1}\sum\limits_{l=k-D}^{k+B-1}\Bigg(\tilde{V}\left(\mathbf{x}^l,\ldots,\mathbf{x}^{l-D}\right)\\
\nonumber&\qquad -\tilde{V}\left(\mathbf{x}^{l+1},\ldots,\mathbf{x}^{l+1-D}\right)\Bigg)\\
\nonumber&=C_3\sum\limits_{k=0}^{T_\epsilon-1}\Bigg(\tilde{V}\left(\mathbf{x}^{k-D},\ldots,\mathbf{x}^{k-2D}\right)\\
\nonumber&\qquad -\tilde{V}\left(\mathbf{x}^{k+B},\ldots,\mathbf{x}^{k+B-D}\right)\Bigg)\\
&\leq C_3(B+D-1)\left(V(\mathbf{x}^0)-\underset{\mathbf{x}\in\mathcal{X}}{\min}\,V(\mathbf{x})\right),
\end{align}
where in (a) we used \eqref{eq:final0_multi} and  defined $C_3$ as \vspace{-0.2cm}$$C_3\triangleq\frac{4\left(2\left(NC_2+1\right)+\frac{DC_2}{3(B+2D-N+1)}\right)}{\gamma\left(2\tau-\gamma L\left(2+D^2\rho^2\right)\right)}.$$
Statement (b) of the theorem   follows readily by defining
\begin{align}
&C_1\triangleq C_3(B+D-1).\label{C1}
\end{align}
\subsection{Proof of Theorem \ref{linear}}\label{proof_linear}
We study now convergence of  Algorithm~\ref{algorithm1} under the additional Assumption B.

First of all, note that one can always find $\eta, \epsilon, \kappa>0$ such that B1 holds. In fact, i) by Lemma \ref{descent}, there exist some $\eta$ and sufficiently small $\gamma/\tau$ such that $V(\mathbf{x}^k)\leq \eta$, for all $k\geq0$; and ii) since $\|\mathbf{x}^k-\texttt{prox}_G\left(\nabla F(\mathbf{x}^k)-\mathbf{x}^k\right)\|_2$ is asymptotically vanishing [Proposition \ref{Prop_best_response}(d) and \eqref{limit2}], one can always find some   $\epsilon>0$ such that $\|\mathbf{x}^k-\texttt{prox}_G\left(\nabla F(\mathbf{x}^k)-\mathbf{x}^k\right)\|_2\leq \epsilon$, for all $k\geq0$.

The proof proceeds along the following steps. \textbf{Step 1:} We first show that the liminf of $\left\{V(\mathbf{x}^k)\right\}$ is a stationary point  $V^\star$, see \eqref{liminf}. \textbf{Step 2} shows that $\left\{V(\mathbf{x}^k)\right\}$ approaches $V^\star$ linearly, up to an error of the order $\mathcal{O}\left(\sum\limits_{l=k-D}^{k+B-1}\|\Delta\widehat{\mathbf{x}}^l_{i_l}\|_2^2\right)$, see \eqref{tseng}. Finally, in \textbf{Step 3} we show that the term $\sum\limits_{l=k-D}^{k+B-1}\|\Delta\widehat{\mathbf{x}}^l_{i_l}\|_2^2$ is overall vanishing at a geometric rate, implying the convergence of $\left\{V(\mathbf{x}^k)\right\}$ to $V^\star$ at a geometric rate.

\subsubsection{\textbf{Step 1}}Pick any vector $\mathbf{x}^\star(\mathbf{x}^k)\in P_{\mathcal{X}^\star}(\mathbf{x}^k)$, where $P_{\mathcal{X}^\star}(\mathbf{x})\triangleq  {\arg\min}_{\mathbf{x}^\star\in\mathcal{X}^\star}\|\mathbf{x}-\mathbf{x}^\star\|_2$, $\mathbf{x}\in\mathbb{R}^n$. Note that:
\begin{equation}d(\mathbf{x}^k\hspace{-.1cm},\mathcal{X}^\star)\hspace{-.1cm}=\hspace{-.1cm}\|\mathbf{x}^\star(\mathbf{x}^k)-\mathbf{x}^k\|_2\hspace{-.05cm}\stackrel{B1}{\leq}\hspace{-.1cm}\kappa\|\mathbf{x}^k-\texttt{prox}_G\hspace{-.1cm}\left(\nabla_{\mathbf{x}}F(\mathbf{x}^k)\hspace{-.1cm}-\hspace{-.1cm}\mathbf{x}^k\right)\hspace{-.1cm}\|_2.\label{errbdnew}\end{equation}
Using \eqref{errbdnew}, \eqref{limit2}, and \eqref{lemma3}, yields
\begin{equation}
\underset{k\rightarrow+\infty}{\lim}\,\|\mathbf{x}^\star(\mathbf{x}^k)-\mathbf{x}^\star(\mathbf{x}^{k+1})\|=0.
\end{equation}
This, together with B2, imply that there exists an index $\bar{k}\geq0$ and a scalar $V^\star$ such that
\begin{equation}
V(\mathbf{x}^\star(\mathbf{x}^k))=V^\star,\quad\forall\,k\geq\bar{k}.\label{kbar_a}
\end{equation}
By the Mean Value Theorem, there exists a vector $\boldsymbol{\xi}^k=\beta^k\mathbf{x}^\star(\mathbf{x}^k)+(1-\beta^k)\mathbf{x}^k$, for some $\beta^k\in(0;1)$, such that, for any $k\geq\bar{k}$,
\begin{align}
\nonumber&V^\star-V(\mathbf{x}^k)=\left\langle\nabla_{\mathbf{x}} F(\boldsymbol{\xi}^k),\mathbf{x}^\star(\mathbf{x}^k)-\mathbf{x}^k\right\rangle+G(\mathbf{x}^\star(\mathbf{x}^k))\\
\nonumber&-G(\mathbf{x}^k)\leq\left\langle\nabla_{\mathbf{x}} F(\boldsymbol{\xi}^k)-\nabla_{\mathbf{x}} F(\mathbf{x}^\star(\mathbf{x}^k)),\mathbf{x}^\star(\mathbf{x}^k)-\mathbf{x}^k\right\rangle\\
\nonumber&\stackrel{(a)}{\leq}\frac{N(\rho^2L^2+1)}{2}\|\mathbf{x}^\star(\mathbf{x}^k)-\mathbf{x}^k\|_2^2\\
&\stackrel{\eqref{lemma3}, \eqref{errbdnew} }{\leq}\frac{N\kappa(\rho^2L^2+1)(1+L+NL_m)}{2}\|\widehat{\mathbf{x}}(\bar{\mathbf{x}})^k-\mathbf{x}^k\|_2,\label{newstr2}
\end{align}	
where (a) follows from A2 and $\|\boldsymbol{\xi}^k-\mathbf{x}^\star(\mathbf{x}^k)\|_2^2
= \|\beta^k \bx^\star(\mathbf{x}^k) + (1-\beta^k) \bx^k-\mathbf{x}^\star(\mathbf{x}^k)\|_2^2
\leq \|\mathbf{x}^\star(\mathbf{x}^k)-\bx^k\|_2^2.$

By invoking \eqref{newstr2}, together with \eqref{limit2}, we obtain
\begin{equation}
\underset{k\rightarrow+\infty}{\lim\inf}\,V(\mathbf{x}^k)\geq V^\star.\label{liminf}
\end{equation}  	

\subsubsection{\textbf{Step 2}}We next show that $V(\mathbf{x}^k)$ approaches $V^\star$ at a linear rate.\\

To this end, consider \eqref{prelyap_multi_02} with $0$ and $1$ replaced by $k$ and $k+1$ respectively; we have the following:
\begin{align}
	&\nonumber V(\mathbf{x}^{k+1})\leq V(\mathbf{x}^k)-\gamma\left(\tau-\gamma L\right)\|\Delta\widehat{\mathbf{x}}_{i^k}\|_2^2\\
	\nonumber&+\frac{L\rho}{2}\sum\limits_{j\in\mathcal{N}_{i^k}}\|\mathbf{x}^k-\mathbf{x}^k(i^k,j)\|_2^2\stackrel{\eqref{x_tilde}}{\leq}V(\mathbf{x}^k)\\&-\gamma\left(\tau-\gamma L\right)\|\Delta\widehat{\mathbf{x}}_{i^k}\|_2^2+\frac{\gamma^2DL\rho^2}{2}\sum\limits_{l=k-D}^{k-1}\|\Delta\widehat{\mathbf{x}}_{i^l}^l\|_2^2\label{temp}.
\end{align}

Is easy to see that, for any $k\geq\bar{k}$, \eqref{temp} implies:
\begin{align}
\nonumber&V(\mathbf{x}^{k+B})-V^\star\leq V(\mathbf{x}^k)-V^\star\\
\nonumber&-\gamma\left(\tau-\frac{\gamma L(2+BD\rho^2)}{2}\right)\sum\limits_{l=k}^{k+B-1}\|\Delta\widehat{\mathbf{x}}_{i^l}^l\|_2^2\\
&+\frac{B\gamma^2DL\rho^2}{2}\sum\limits_{l=k-D}^{k-1}\|\Delta\widehat{\mathbf{x}}_{i^l}^l\|_2^2.
\label{descent_bis}
\end{align}

To prove the desired result we will combine next \eqref{descent_bis} with the following lemma.
\begin{lemma}\label{lemmarate}
	For any $k\geq0$, there holds:
	\begin{align}
	\nonumber&V(\mathbf{x}^{k+B})-V(\mathbf{x}^\star(\mathbf{x}^k))\leq(1-\gamma)\left(V(\mathbf{x}^k)-V(\mathbf{x}^\star(\mathbf{x}^k))\right)\\
	\label{rate}&+\gamma\left(N\alpha_1+(B-N)\alpha_2\right)\sum\limits_{l=k-D}^{k+B-1}\|\Delta\widehat{\mathbf{x}}_{i^l}^l\|_2^2,
	\end{align}
	where $\alpha_1$ and $\alpha_2$ are two positive constants defined in Appendix \ref{suppl} [see \eqref{alpha1} and \eqref{alpha2}, respectively].
\end{lemma}
\begin{proof}
	See Section \ref{suppl}.
\end{proof}

Multiplying the two sides of \eqref{descent_bis} and \eqref{rate} by $(N\alpha_1+(B-N)\alpha_2)$ and $\tau-\gamma L(2+BD\rho^2)/2$ respectively, and adding the two inequalities together, yields
\begin{align}
\label{tseng}&V(\mathbf{x}^{k+B})-V^\star\leq\theta\left(V(\mathbf{x}^k)-V^\star\right)+\zeta\sum\limits_{l=k-D}^{k-1}\|\Delta\widehat{\mathbf{x}}_{i^l}^l\|_2^2,
\end{align}
for all $k\geq\bar{k}$, where
\begin{align*}
\theta\hspace{-.1cm}\triangleq\frac{(1-\gamma)(2\tau-\gamma L(BD\mathcal{N}^2_m+2))+2N\alpha_1+2(B-N)\alpha_2}{2\tau-\gamma L(BD\mathcal{N}^2_m+2)+2N\alpha_1+2(B-N)\alpha_2},
\end{align*}
and
\begin{align*}
\zeta\hspace{-.1cm}\triangleq\frac{(N\alpha_1+(B-N)\alpha_2)(2\tau+\gamma L(BD\rho^2(\gamma-1)+2))}{2\tau-\gamma L(BD\mathcal{N}^2_m+2)+2N\alpha_1+2(B-N)\alpha_2}.
\end{align*}
\subsubsection{\textbf{Step 3}}We can now apply Lemma 4.5 in \cite{tseng1991rate} by noticing that \eqref{descent_bis}, \eqref{liminf}, and \eqref{tseng} correspond, respectively, to (4.21), (4.22), and to the first inequality after (4.23) in \cite{tseng1991rate}. Theorem \ref{linear} readily follows, setting\begin{align}
\nonumber\lambda\triangleq&1-\frac{\gamma^2}{2}\frac{2\tau-\gamma L(BD\rho^2+2)}{2N\alpha_1+2(B-N)\alpha_2}\\\label{lambdalab}
&\cdot\frac{}{+\gamma(2-\gamma)(2\tau-\gamma L(BD\rho^2+2))}.
\end{align}
\subsection{Miscellanea results}\label{suppl}
This section contains the proofs of Lemma \ref{tworesults} and Lemma \ref{lemmarate}.\\

\textbf{\textit{{Proof of Lemma \ref{tworesults}}:}}
	(i) Assume without loss of generality that $t\leq h$. We have
	\begin{align}
	\nonumber&\|\widehat{\mathbf{x}}_i(\tilde{\mathbf{x}}^t(i))-\widehat{\mathbf{x}}_i(\tilde{\mathbf{x}}^h(i))\|_2^2\stackrel{\eqref{lipsch}}{\leq}\frac{\rho L_m^2}{\tau^2}\|\mathbf{x}^t(i,i)-\mathbf{x}^h(i,i)\|_2^2\\&\nonumber+\frac{\rho L^2}{\tau^2}\sum\limits_{j\in\mathcal{N}_i\backslash\{i\}}\|\mathbf{x}^t(i,j)-\mathbf{x}^h(i,j)\|_2^2\\
	\nonumber&\stackrel{\eqref{x_tilde}, \eqref{update}}{\leq}\left(\frac{3\rho\left(L_m^2+(\rho-1)L^2\right)}{\tau^2}\right)\Bigg(\gamma^2(B-N+1)\\
	\nonumber&\sum\limits_{l=t}^{h-1}\|\Delta\widehat{\mathbf{x}}_{i^l}^l\|_2^2+D\gamma^2\left(\sum\limits_{l=t-D}^{t-1}\|\Delta\widehat{\mathbf{x}}_{i^l}^l\|^2_2+\sum\limits_{l=h-D}^{h-1}\|\Delta\widehat{\mathbf{x}}_{i^l}^l\|_2^2\right)\Bigg).
	\end{align}
	(ii) Define $r_i^{h,k}\triangleq\underset{t\in[k;k+B-1]:i^t=i}{\arg\min}\,|t-h|$. We have:
	\begin{align}
	\|\Delta\widehat{\mathbf{x}}^h\|_2^2&\leq2\sum\limits_{i=1}^N\left(\|\widehat{\mathbf{x}}_i^h-\widehat{\mathbf{x}}_i^{r_i^{h,k}}\|_2^2+\|\Delta\widehat{\mathbf{x}}_i^{r^{h,k}_i}\|_2^2\right)\\
	&\nonumber\stackrel{\eqref{lip}}{\leq}2\sum\limits_{i=1}^N\Bigg(C_2\sum\limits_{l=k-D}^{k+B-2}\|\Delta\widehat{\mathbf{x}}_{i^l}^l\|_2^2+\|\Delta\widehat{\mathbf{x}}_i^{r_i^{h,k}}\|_2^2\Bigg).
	\end{align}	
	
\textbf{\textit{{Proof of Lemma \ref{lemmarate}}:}} Define  $T_i^k+1$ as the number of times agent $i$ performs its update within $[k,k+B-1]$; let  $l^k_{i,0},\ldots,l^k_{i,T_i^k}$,be the iteration indexes of such updates. By the Mean Value Theorem, there exists a vector $\boldsymbol{\xi}^k=\beta^k {\bx}^\star(\bx^k) + (1-\beta^k) \mathbf{x}^k$, for some  $\beta^k\in (0,1)$, such that
\begin{align}
\nonumber& V(\mathbf{x}^{k+B})-V(\mathbf{x}^\star(\mathbf{x}^k))=\left\langle\nabla_{\mathbf{x}}F(\boldsymbol{\xi}^k),\mathbf{x}^{k+B}-\mathbf{x}^\star(\mathbf{x}^k)\right\rangle\\\nonumber&+G(\mathbf{x}^{k+B})-G(\mathbf{x}^\star(\mathbf{x}^k))\\
\nonumber
&=\sum\limits_{i=1}^N\Bigg(\underbrace{\left\langle\nabla_{\mathbf{x}_i}F(\boldsymbol{\xi}^k),\mathbf{x}_i^{l^k_{i,1}}-\mathbf{x}_i^\star(\mathbf{x}^k)\right\rangle}_{\texttt{term II}}\\
\nonumber&\quad+\sum\limits_{t=1}^{T_i^k-1}\underbrace{\left\langle\nabla_{\mathbf{x}_i}F(\boldsymbol{\xi}^k),\mathbf{x}^{l^k_{i,t+1}}_i-\mathbf{x}_i^{l^k_{i,t}}\right\rangle}_{\texttt{term III}}\\
\nonumber&\quad+\underbrace{\left\langle\nabla_{\mathbf{x}_i}F(\boldsymbol{\xi}^k),\mathbf{x}^{k+B}_i-\mathbf{x}_i^{l^k_{i,T_i^k}}\right\rangle}_{\texttt{term IV}}\Bigg)\\
\label{begin_a}&\quad+G(\mathbf{x}^{k+B})-G(\mathbf{x}^\star(\mathbf{x}^k)).
\end{align}

To prove \eqref{rate}, it is then sufficient show that  \texttt{term  II}, \texttt{term III}, and \texttt{term IV} in \eqref{begin_a} converge at a geometric rate up to an error of the order $\mathcal{O}\left(\sum\limits_{l=k-D}^{k+B-1}\|\Delta\widehat{\mathbf{x}}^l_{i_l}\|_2^2\right)$. To do this, we first show that \texttt{term  II}, \texttt{term III}, and \texttt{term IV} converges at a geometric rate up to the error terms $a_{i,4}^k$, $b_{i,t,4}^k$, and $c_{i,4}^k$, respectively [see \eqref{one_a}, \eqref{two_a}, and \eqref{three_a}]. Then, we prove that each of these errors is of the order $\mathcal{O}\left(\sum\limits_{l=k-D}^{k+B-1}\|\Delta\widehat{\mathbf{x}}^l_{i_l}\|_2^2\right)$, as desired [see \eqref{aik}, and \eqref{bcik}].

\texttt{Term II} can be upper bounded as
\begin{align}
\nonumber&\left\langle\nabla_{\mathbf{x}_i}F(\boldsymbol{\xi}^k),\mathbf{x}^{l^k_{i,1}}_i-\mathbf{x}_i^\star(\mathbf{x}^k)\right\rangle\\
\nonumber&\stackrel{A2}{\leq}\left\langle\nabla_{\mathbf{x}_i}F\left(\widehat{\mathbf{x}}^{l^k_{i,0}}\right),\mathbf{x}^{l^k_{i,1}}_i-\mathbf{x}_i^\star(\mathbf{x}^k)\right\rangle\\
\nonumber&\quad+\underbrace{\rho L\left\|\widehat{\mathbf{x}}^{l^k_{i,0}}-\boldsymbol{\xi}^k\right\|_2\left\|\mathbf{x}^{l^k_{i,1}}_i-\mathbf{x}_i^\star(\mathbf{x}^k)\right\|_2}_{\triangleq a_{i,1}^k}\\
\nonumber&\hspace{-.35cm}\stackrel{A2,C2,C3}{\leq}\Bigg\langle\nabla\tilde{f}_i\left(\widehat{\mathbf{x}}_i^{l^k_{i,0}};\mathbf{x}^{l^k_{i,0}}_{\mathcal{N}_i}(i,i)\right)\\
\nonumber&\quad+\sum\limits_{j\in\mathcal{N}_i\backslash\{i\}}\nabla_{\mathbf{x}_i}f_j\left(\mathbf{x}^{l_{i,0}^k}_{\mathcal{N}_j}(i,j)\right),
\mathbf{x}_i^{l^k_{i,1}}-\mathbf{x}_i^\star(\mathbf{x}^k)\Bigg\rangle\\
\nonumber&\quad+\left\|\mathbf{x}^{l^k_{i,1}}_i-\mathbf{x}_i^\star(\mathbf{x}^k)\right\|_2\Bigg(L_i\left\|\widehat{\mathbf{x}}^{l_{i,0}^k}_{\mathcal{N}_i}-\mathbf{x}^{l^k_{i,0}}_{\mathcal{N}_i}(i,i)\right\|_2\\
\nonumber&\quad+L\sum\limits_{j\in\mathcal{N}_i\backslash\{i\}}\left\|\widehat{\mathbf{x}}_{\mathcal{N}_j}^{l_{i,0}^k}-\mathbf{x}^{l_{i,0}^k}_{\mathcal{N}_j}(i,j)\right\|_2\Bigg)+a_{i,1}^k\\
\nonumber&\stackrel{(a)}{\leq}(\gamma-1)\Bigg\langle\nabla\tilde{f}_i\left(\widehat{\mathbf{x}}_i^{l^k_{i,0}};\mathbf{x}^{l^k_{i,0}}_{\mathcal{N}_i}(i,i)\right)\\
\nonumber&\quad+\sum\limits_{j\in\mathcal{N}_i\backslash\{i\}}\nabla_{\mathbf{x}_i}f_j\left(\mathbf{x}^{l_{i,0}^k}_{\mathcal{N}_j}(i,j)\right),\Delta\widehat{\mathbf{x}}_i^{l^k_{i,0}}\Bigg\rangle+g_i\left(\mathbf{x}_i^\star(\mathbf{x}^k)\right)\\
\nonumber&\quad-g_i\left(\widehat{\mathbf{x}}_i^{l^k_{i,0}}\right)+a_{i,2}^k\stackrel{C2}{\leq} g_i(\mathbf{x}_i^\star(\mathbf{x}^k))-g_i\left(\widehat{\mathbf{x}}_i^{l^k_{i,0}}\right)\\
\nonumber&+(\gamma-1)\left\langle\sum\limits_{j\in\mathcal{N}_i}\nabla_{\mathbf{x}_i}f_j\left(\mathbf{x}^{l^k_{i,0}}_{\mathcal{N}_j}(i,j)\right),\Delta\widehat{\mathbf{x}}^{l^k_{i,0}}_i\right\rangle\\
\nonumber&+(1-\gamma)\Bigg\|\nabla\tilde{f}_i\left(\widehat{\mathbf{x}}_i^{l^k_{i,0}};\mathbf{x}^{l^k_{i,0}}_{\mathcal{N}_i}(i,i)\right)-\nabla\tilde{f}_i\left(\mathbf{x}_i^{l^k_{i,0}};\mathbf{x}^{l^k_{i,0}}_{\mathcal{N}_i}(i,i)\right)\Bigg\|_2\\
\nonumber&\cdot\left\|\Delta\widehat{\mathbf{x}}^{l^k_{i,0}}_i\right\|_2+a_{i,2}^k\stackrel{(b)}{\leq}g_i\left(\mathbf{x}_i^\star(\mathbf{x}^k)\right)-g_i\left(\widehat{\mathbf{x}}_i^{l^k_{i,0}}\right)\\
\nonumber&+\frac{1-\gamma}{\gamma}\left(V\left(\mathbf{x}^{l^k_{i,0}}\right)-V\left(\mathbf{x}^{l^k_{i,0}+1}\right)\right)\\
\nonumber&+(1-\gamma)\left\|\sum\limits_{j\in\mathcal{N}_i}\left(\nabla_{\mathbf{x}_i}f_j\left(\mathbf{x}^{l^k_{i,0}}_{\mathcal{N}_j}\right)-\nabla_{\mathbf{x}_i}f_j\left(\mathbf{x}^{l^k_{i,0}}_{\mathcal{N}_j}(i,j)\right)\right)\right\|_2\\
\nonumber&\cdot\left\|\Delta\widehat{\mathbf{x}}_i^{l^k_{i,0}}\right\|_2+\frac{L\gamma(1-\gamma)}{2}\left\|\Delta\widehat{\mathbf{x}}_i^{l^k_{i,0}}\right\|_2^2+a_{i,3}^k\\
\nonumber&+(1-\gamma)\left(g_i\left(\widehat{\mathbf{x}}_i^{l^k_{i,0}}\right)-g_i\left(\mathbf{x}_i^{l^k_{i,0}}\right)\right)\\
\nonumber&\stackrel{(c)}{=}\frac{1-\gamma}{\gamma}\left(V\left(\mathbf{x}^{l^k_{i,0}}\right)-V\left(\mathbf{x}^{l^k_{i,0}+1}\right)\right)+g_i\left(\mathbf{x}_i^\star(\mathbf{x}^k)\right)\\
\label{one_a}&\quad+(\gamma-1)g_i\left(\mathbf{x}_i^{l^k_{i,0}}\right)
-\gamma g_i\left(\widehat{\mathbf{x}}_i^{l^k_{i,0}}\right)+a_{i,4}^k;
\end{align}
where the quantities $a_{i,2}^k$ in (a), and $a_{i,3}^k$ in (b) are defined in \eqref{ai2} and \eqref{ai3} at the bottom of the next page, respectively; furthermore in (b) we used the descent lemma, and in (c) we defined\begin{align*}
&a_{i,4}^k\triangleq \,a_{i,3}^k+\frac{L\gamma(1-\gamma)}{2}\left\|\Delta\widehat{\mathbf{x}}_i^{l^k_{i,0}}\right\|_2^2+(1-\gamma)\\\nonumber&\cdot\hspace{-.05cm}\underbrace{\left\|\sum\limits_{j\in\mathcal{N}_i}\left(\nabla_{\mathbf{x}_i}f_j\left(\mathbf{x}^{l^k_{i,0}}_{\mathcal{N}_j}\right)-\nabla_{\mathbf{x}_i}f_j\left(\mathbf{x}^{l^k_{i,0}}_{\mathcal{N}_j}(i,j)\right)\right)\right\|_2\left\|\Delta\widehat{\mathbf{x}}_i^{l^k_{i,0}}\right\|_2}_{\texttt{term VII}}\hspace{-.05cm}.
\end{align*} \begin{figure*}[b]\hrule\begin{align}
&a_{i,2}^k\triangleq a_{i,1}^k+\underbrace{\left\|\mathbf{x}^{l^k_{i,1}}_i-\mathbf{x}_i^\star(\mathbf{x}^k)\right\|_2\left(L_i\left\|\widehat{\mathbf{x}}^{l_{i,0}^k}_{\mathcal{N}_i}-\mathbf{x}^{l^k_{i,0}}_{\mathcal{N}_i}(i,i)\right\|_2+L\sum\limits_{j\in\mathcal{N}_i\backslash\{i\}}\left\|\widehat{\mathbf{x}}_{\mathcal{N}_j}^{l_{i,0}^k}-\mathbf{x}^{l_{i,0}^k}_{\mathcal{N}_j}(i,j)\right\|_2\right)}_{\texttt{term V}}\label{ai2}\\
&a_{i,3}^k\triangleq a_{i,2}^k+(1-\gamma)\underbrace{\bigg\|\nabla\tilde{f}_i\left(\widehat{\mathbf{x}}_i^{l^k_{i,0}};\mathbf{x}^{l^k_{i,0}}_{\mathcal{N}_i}(i,i)\right)-\nabla\tilde{f}_i\left(\mathbf{x}_i^{l^k_{i,0}};\mathbf{x}^{l^k_{i,0}}_{\mathcal{N}_i}(i,i)\right)\bigg\|_2\left\|\Delta\widehat{\mathbf{x}}^{l^k_{i,0}}_i\right\|_2}_{\texttt{term VI}}\label{ai3}\end{align}\end{figure*}

\texttt{Term III} can be upper bounded as: for any $i$ and $t\in[1,T^k_i-1]$,
\begin{align}
\nonumber&\left\langle\nabla_{\mathbf{x}_i}F(\boldsymbol{\xi}^k),\mathbf{x}^{l^k_{i,t+1}}_i-\mathbf{x}_i^{l^k_{i,t}}\right\rangle\\
\nonumber&\stackrel{A2}{\leq}\left\langle\nabla_{\mathbf{x}_i}F\left(\widehat{\mathbf{x}}^{l^k_{i,t}}\right),\mathbf{x}^{l^k_{i,t+1}}_i-\mathbf{x}_i^{l^k_{i,t}}\right\rangle\\
\nonumber&\quad+\underbrace{\rho L\left\|\widehat{\mathbf{x}}^{l^k_{i,t}}-\boldsymbol{\xi}^k\right\|_2\left\|\mathbf{x}^{l^k_{i,t}}_i-\mathbf{x}_i^{l^k_{i,t+1}}\right\|_2}_{\triangleq b_{i,t,1}^k}\\
\nonumber&\hspace{-.35cm}\stackrel{A2,C2, C3}{\leq}\Bigg\langle\nabla\tilde{f}_i\left(\widehat{\mathbf{x}}_i^{l^k_{i,t}};\mathbf{x}^{l^k_{i,t}}_{\mathcal{N}_i}(i,i)\right)\\
\nonumber&\quad+\sum\limits_{j\in\mathcal{N}_i\backslash\{i\}}\nabla_{\mathbf{x}_i}f_j\left(\mathbf{x}_{\mathcal{N}_j}^{l_{i,t}^k}(i,j)\right),\mathbf{x}_i^{l^k_{i,t+1}}-\mathbf{x}_i^{l^k_{i,t}}\Bigg\rangle\\
\nonumber&\quad+\left\|\mathbf{x}^{l^k_{i,t}}_i-\mathbf{x}_i^{l^k_{i,t+1}}\right\|_2\Bigg(L_i\left\|\widehat{\mathbf{x}}^{l_{i,t}^k}_{\mathcal{N}_i}-\mathbf{x}^{l^k_{i,t}}_{\mathcal{N}_i}(i,i)\right\|_2\\
\nonumber&\quad+L\sum\limits_{j\in\mathcal{N}_i\backslash\{i\}}\left\|\widehat{\mathbf{x}}^{l_{i,t}^k}_{\mathcal{N}_j}-\mathbf{x}^{l_{i,t}^k}_{\mathcal{N}_j}(i,j)\right\|_2\Bigg)+b_{i,t,1}^k\\
\nonumber&\stackrel{(a)}{\leq}(\gamma-1)\Bigg\langle\nabla\tilde{f}_i\left(\widehat{\mathbf{x}}_i^{l^k_{i,t}};\mathbf{x}^{l^k_{i,t}}_{\mathcal{N}_i}(i,i)\right)\\
\nonumber&\quad+\sum\limits_{j\in\mathcal{N}_i\backslash\{i\}}\nabla_{\mathbf{x}_i}f_j\left(\mathbf{x}_{\mathcal{N}_j}^{l_{i,t}^k}(i,j)\right),\Delta\widehat{\mathbf{x}}_i^{l^k_{i,t}}\Bigg\rangle+g_i\left(\mathbf{x}_i^{l^k_{i,t}}\right)\\
\nonumber&\quad-g_i\left(\widehat{\mathbf{x}}_i^{l^k_{i,t}}\right)+b_{i,t,2}^k \stackrel{C2}{\leq}g_i\left(\mathbf{x}_i^{l^k_{i,t}}\right)-g_i\left(\widehat{\mathbf{x}}_i^{l^k_{i,t}}\right)\\
\nonumber&+(\gamma-1)\left\langle\sum\limits_{j\in\mathcal{N}_i}\nabla_{\mathbf{x}_i}f_j\left(\mathbf{x}_{\mathcal{N}_j}^{l^k_{i,t}}(i,j)\right),\Delta\widehat{\mathbf{x}}^{l^k_{i,t}}_i\right\rangle\\
\nonumber&+(1-\gamma)\left\|\nabla\tilde{f}_i\left(\widehat{\mathbf{x}}_i^{l^k_{i,t}};\mathbf{x}^{l^k_{i,t}}_{\mathcal{N}_i}(i,i)\right)-\nabla\tilde{f}_i\left(\mathbf{x}_i^{l^k_{i,t}};\mathbf{x}^{l^k_{i,t}}_{\mathcal{N}_i}(i,i)\right)\right\|_2\\
\nonumber&\cdot\left\|\Delta\widehat{\mathbf{x}}^{l^k_{i,t}}_i\right\|_2+b_{i,t,2}^k\stackrel{(b)}{\leq}g_i\left(\mathbf{x}_i^{l^k_{i,t}}\right)-g_i\left(\widehat{\mathbf{x}}_i^{l^k_{i,t}}\right)\\
\nonumber&+\frac{1-\gamma}{\gamma}\left(V\left(\mathbf{x}^{l^k_{i,t}}\right)-V\left(\mathbf{x}^{l^k_{i,t}+1}\right)\right)\\
\nonumber&+(1-\gamma)\left\|\sum\limits_{j\in\mathcal{N}_i}\left(\nabla_{\mathbf{x}_i}f_j\left(\mathbf{x}^{l^k_{i,t}}_{\mathcal{N}_j}\right)-\nabla_{\mathbf{x}_i}f_j\left(\mathbf{x}^{l^k_{i,t}}_{\mathcal{N}_j}(i,j)\right)\right)\right\|_2\\
\nonumber&\left\|\Delta\widehat{\mathbf{x}}_i^{l^k_{i,t}}\right\|_2+\frac{L\gamma(1-\gamma)}{2}\left\|\Delta\widehat{\mathbf{x}}_i^{l^k_{i,t}}\right\|_2^2+b_{i,t,3}^k\\
\nonumber&+(1-\gamma)\left(g_i\left(\widehat{\mathbf{x}}_i^{l^k_{i,t}}\right)-g_i\left(\mathbf{x}_i^{l^k_{i,t}}\right)\right)\\
\nonumber&=\frac{1-\gamma}{\gamma}\left(V\left(\mathbf{x}^{l^k_{i,t}}\right)-V\left(\mathbf{x}^{l^k_{i,t}+1}\right)\right)\\
\label{two_a}&\quad+\gamma\left(g_i\left(\mathbf{x}_i^{l^k_{i,t}}\right)-g_i\left(\widehat{\mathbf{x}}_i^{l^k_{i,t}}\right)\right)+b_{i,t,4}^k;
\end{align}
where the quantities $b_{i,t,2}^k$ in (a), and $b_{i,t,3}^k$ in (b) are defined in \eqref{bit2} and \eqref{bit3} at the bottom of the next page, respectively; furthermore in (b) we used the descent lemma, and in (c) we defined \begin{align*}
&b_{i,t,4}^k\triangleq\,b_{i,t,3}^k+\frac{L\gamma(1-\gamma)}{2}\left\|\Delta\widehat{\mathbf{x}}_i^{l^k_{i,t}}\right\|_2^2+(1-\gamma)\\
&\cdot\hspace{-.05cm}\underbrace{\left\|\sum\limits_{j\in\mathcal{N}_i}\left(\nabla_{\mathbf{x}_i}f_j\left(\mathbf{x}^{l^k_{i,t}}_{\mathcal{N}_j}\right)-\nabla_{\mathbf{x}_i}f_j\left(\mathbf{x}^{l^k_{i,t}}_{\mathcal{N}_j}(i,j)\right)\right)\right\|_2\left\|\Delta\widehat{\mathbf{x}}_i^{l^k_{i,t}}\right\|_2}_{\texttt{term VII}}\hspace{-.05cm}.\end{align*}

\begin{figure*}[b]\hrule\begin{align}
&b_{i,t,2}^k\triangleq b_{i,t,1}^k+\underbrace{\left\|\mathbf{x}^{l^k_{i,t}}_i-\mathbf{x}_i^{l^k_{i,t+1}}\right\|_2\left(L_i\left\|\widehat{\mathbf{x}}^{l_{i,t}^k}_{\mathcal{N}_i}-\mathbf{x}^{l^k_{i,t}}_{\mathcal{N}_i}(i,i)\right\|_2+L\sum\limits_{j\in\mathcal{N}_i\backslash\{i\}}\left\|\widehat{\mathbf{x}}^{l_{i,t}^k}_{\mathcal{N}_j}-\mathbf{x}^{l_{i,t}^k}_{\mathcal{N}_j}(i,j)\right\|_2\right)}_{\texttt{term VIII}}\label{bit2}\\
&b_{i,t,3}^k\triangleq b_{i,t,2}^k+(1-\gamma)\underbrace{\left\|\nabla\tilde{f}_i\left(\widehat{\mathbf{x}}_i^{l^k_{i,t}};\mathbf{x}^{l^k_{i,t}}_{\mathcal{N}_i}(i,i)\right)-\nabla\tilde{f}_i\left(\mathbf{x}_i^{l^k_{i,t}};\mathbf{x}^{l^k_{i,t}}_{\mathcal{N}_i}(i,i)\right)\right\|_2\left\|\Delta\widehat{\mathbf{x}}^{l^k_{i,t}}_i\right\|_2}_{\texttt{term VI}}\label{bit3}\end{align}\end{figure*}
Following similar steps, we can bound \texttt{term IV}, as
\begin{align}
\nonumber&\left\langle\nabla_{\mathbf{x}_i}F(\boldsymbol{\xi}^k),\mathbf{x}^{k+B}_i-\mathbf{x}_i^{l^k_{i,T_i^k}}\right\rangle\\
\nonumber&\stackrel{A2}{\leq}\left\langle\nabla_{\mathbf{x}_i}F\left(\widehat{\mathbf{x}}^{l^k_{i,T^k_i}}\right),\mathbf{x}^{k+B}_i-\mathbf{x}_i^{l^k_{i,T^k_i}}\right\rangle\\
\nonumber&\quad+\underbrace{\rho L\left\|\widehat{\mathbf{x}}^{l^k_{i,T^k_i}}-\boldsymbol{\xi}^k\right\|_2\left\|\mathbf{x}^{l^k_{i,T^k_i}}_i-\mathbf{x}_i^{k+B}\right\|_2}_{c_{i,1}^k}\\
\nonumber&\hspace{-.35cm}\stackrel{A2,C2, C3}{\leq}\Bigg\langle\nabla\tilde{f}_i\left(\widehat{\mathbf{x}}_i^{l^k_{i,T^k_i}};\mathbf{x}^{l^k_{i,T^k_i}}_{\mathcal{N}_i}(i,i)\right)\\
\nonumber&\quad+\sum\limits_{j\in\mathcal{N}_i\backslash\{i\}}\nabla_{\mathbf{x}_i}f_j\left(\mathbf{x}_{\mathcal{N}_j}^{l_{i,T^k_i}^k}(i,j)\right),\mathbf{x}_i^{k+B}-\mathbf{x}_i^{l^k_{i,T^k_i}}\Bigg\rangle\\
\nonumber&\quad+\left\|\mathbf{x}^{l^k_{i,T^k_i}}_i-\mathbf{x}_i^{k+B}\right\|_2\Bigg(L_i\left\|\widehat{\mathbf{x}}^{l_{i,T^k_i}^k}_{\mathcal{N}_i}-\mathbf{x}^{l^k_{i,T^k_i}}_{\mathcal{N}_i}(i,i)\right\|_2\\
\nonumber&\quad+L\sum\limits_{j\in\mathcal{N}_i\backslash\{i\}}\left\|\widehat{\mathbf{x}}^{l_{i,T^k_i}^k}_{\mathcal{N}_j}-\mathbf{x}^{l_{i,T^k_i}^k}_{\mathcal{N}_j}(i,j)\right\|_2\Bigg)+c_{i,1}^k\\
\nonumber&\stackrel{(a)}{\leq}(\gamma-1)\Bigg\langle\nabla\tilde{f}_i\left(\widehat{\mathbf{x}}_i^{l^k_{i,T^k_i}};\mathbf{x}^{l^k_{i,T^k_i}}_{\mathcal{N}_i}(i,i)\right)\\
\nonumber&\quad+\sum\limits_{j\in\mathcal{N}_i\backslash\{i\}}\nabla_{\mathbf{x}_i}f_j\left(\mathbf{x}_{\mathcal{N}_j}^{l_{i,T^k_i}^k}(i,j)\right),\Delta\widehat{\mathbf{x}}_i^{l^k_{i,T^k_i}}\Bigg\rangle+g_i\left(\mathbf{x}_i^{l^k_{i,T^k_i}}\right)\\
\nonumber&-g_i\left(\widehat{\mathbf{x}}_i^{l^k_{i,T^k_i}}\right)+c_{i,2}^k \stackrel{C2}{\leq}g_i\left(\mathbf{x}_i^{l^k_{i,T^k_i}}\right)-g_i\left(\widehat{\mathbf{x}}_i^{l^k_{i,T^k_i}}\right)\\
\nonumber&+(\gamma-1)\left\langle\sum\limits_{j\in\mathcal{N}_i}\nabla_{\mathbf{x}_i}f_j\left(\mathbf{x}_{\mathcal{N}_j}^{l^k_{i,T^k_i}}(i,j)\right),\Delta\widehat{\mathbf{x}}^{l^k_{i,T^k_i}}_i\right\rangle\\
\nonumber&+(1-\gamma)\Bigg\|\nabla\tilde{f}_i\left(\widehat{\mathbf{x}}_i^{l^k_{i,T^k_i}};\mathbf{x}^{l^k_{i,T^k_i}}_{\mathcal{N}_i}(i,i)\right)-\\
\nonumber&\nabla\tilde{f}_i\left(\mathbf{x}_i^{l^k_{i,T^k_i}};\mathbf{x}^{l^k_{i,T^k_i}}_{\mathcal{N}_i}(i,i)\right)\Bigg\|_2\left\|\Delta\widehat{\mathbf{x}}^{l^k_{i,T^k_i}}_i\right\|_2+c_{i,2}^k\stackrel{(b)}{\leq}g_i\left(\mathbf{x}_i^{l^k_{i,T^k_i}}\right)\\
\nonumber&-g_i\left(\widehat{\mathbf{x}}_i^{l^k_{i,T^k_i}}\right)+\frac{1-\gamma}{\gamma}\left(V\left(\mathbf{x}^{l^k_{i,T^k_i}}\right)-V\left(\mathbf{x}^{l^k_{i,T^k_i}+1}\right)\right)\\
\nonumber&+(1-\gamma)\left\|\sum\limits_{j\in\mathcal{N}_i}\left(\nabla_{\mathbf{x}_i}f_j\left(\mathbf{x}^{l^k_{i,T^k_i}}_{\mathcal{N}_j}\right)-\nabla_{\mathbf{x}_i}f_j\left(\mathbf{x}^{l^k_{i,T^k_i}}_{\mathcal{N}_j}(i,j)\right)\right)\right\|_2\\
\nonumber&\left\|\Delta\widehat{\mathbf{x}}_i^{l^k_{i,T^k_i}}\right\|_2+\frac{L\gamma(1-\gamma)}{2}\left\|\Delta\widehat{\mathbf{x}}_i^{l^k_{i,T^k_i}}\right\|_2^2+c_{i,3}^k\\
\nonumber&+(1-\gamma)\left(g_i\left(\widehat{\mathbf{x}}_i^{l^k_{i,T^k_i}}\right)-g_i\left(\mathbf{x}_i^{l^k_{i,T^k_i}}\right)\right)\\
\nonumber&=\frac{1-\gamma}{\gamma}\left(V\left(\mathbf{x}^{l^k_{i,T^k_i}}\right)-V\left(\mathbf{x}^{l^k_{i,T^k_i}+1}\right)\right)\\
\label{three_a}&\quad+\gamma\left(g_i\left(\mathbf{x}_i^{l^k_{i,T^k_i}}\right)-g_i\left(\widehat{\mathbf{x}}_i^{l^k_{i,T^k_i}}\right)\right)+c_{i,4}^k;
\end{align}
where the quantities $c_{i,2}^k$ in (a), and $c_{i,3}^k$ in (b) are defined in \eqref{ci2} and \eqref{ci3} at the bottom of the next page, respectively; furthermore in (b) we used the descent lemma, and in (c) we defined  \begin{align*}
&c_{i,4}^k\triangleq\, c_{i,3}^k+\frac{L\gamma(1-\gamma)}{2}\left\|\Delta\widehat{\mathbf{x}}_i^{l^k_{i,T^k_i}}\right\|_2^2+(1-\gamma)\\
&\hspace{-.05cm}\cdot\hspace{-.05cm}\underbrace{\left\|\sum\limits_{j\in\mathcal{N}_i}\hspace{-.1cm}\left(\nabla_{\mathbf{x}_i}f_j\left(\mathbf{x}^{l^k_{i,T^k_i}}_{\mathcal{N}_j}\right)-\nabla_{\mathbf{x}_i}f_j\left(\mathbf{x}^{l^k_{i,T^k_i}}_{\mathcal{N}_j}(i,j)\right)\right)\right\|_2\hspace{-.05cm}\left\|\Delta\widehat{\mathbf{x}}_i^{l^k_{i,T^k_i}}\right\|_2}_{\texttt{term VII}}\hspace{-.1cm}.
\end{align*}

 \begin{figure*}[b]\begin{align}&c_{i,2}^k\triangleq c_{i,1}^k+\underbrace{\left\|\mathbf{x}^{l^k_{i,T^k_i}}_i-\mathbf{x}_i^{k+B}\right\|_2\left(L_i\left\|\widehat{\mathbf{x}}^{l_{i,T^k_i}^k}_{\mathcal{N}_i}-\mathbf{x}^{l^k_{i,T^k_i}}_{\mathcal{N}_i}(i,i)\right\|_2+L\sum\limits_{j\in\mathcal{N}_i\backslash\{i\}}\left\|\widehat{\mathbf{x}}^{l_{i,T^k_i}^k}_{\mathcal{N}_j}-\mathbf{x}^{l_{i,T^k_i}^k}_{\mathcal{N}_j}(i,j)\right\|_2\right)}_{\texttt{term IX}}\label{ci2}\\
	&c_{i,3}^k\triangleq c_{i,2}^k+(1-\gamma)\underbrace{\left\|\nabla\tilde{f}_i\left(\widehat{\mathbf{x}}_i^{l^k_{i,T^k_i}};\mathbf{x}^{l^k_{i,T^k_i}}_{\mathcal{N}_i}(i,i)\right)-\nabla\tilde{f}_i\left(\mathbf{x}_i^{l^k_{i,T^k_i}};\mathbf{x}^{l^k_{i,T^k_i}}_{\mathcal{N}_i}(i,i)\right)\right\|_2\left\|\Delta\widehat{\mathbf{x}}^{l^k_{i,T^k_i}}_i\right\|_2}_{\texttt{term VI}}\label{ci3}\end{align}\end{figure*}

\noindent We now show that the error terms $a_{i,4}^k$, $b_{i,t,4}^k$, and $c_{i,4}^k$, are of the order $\mathcal{O}\left(\sum\limits_{l=k-D}^{k+B-1}\|\Delta\widehat{\mathbf{x}}^l_{i_l}\|_2^2\right)$. To do so, in the following we properly upper bound each term inside $a_{i,4}^k$, $b_{i,t,4}^k$, and $c_{i,4}^k$.

We begin noticing that, by the definition of $\boldsymbol{\xi}^k$, it follows\vspace{-0.3cm}
\begin{align}
\nonumber&\|\widehat{\mathbf{x}}^h-\boldsymbol{\xi}^k\|_2\\
\nonumber&=\|(1-\beta^k)\mathbf{x}^k+\beta^k\mathbf{x}^\star(\mathbf{x}^k)-\widehat{\mathbf{x}}^h\|_2\\
\nonumber&\leq\|\mathbf{x}^k-\mathbf{x}^\star(\mathbf{x}^k)\|_2+\|\widehat{\mathbf{x}}^h-\mathbf{x}^k\|_2\\
&\leq\|\mathbf{x}^k-\mathbf{x}^\star(\mathbf{x}^k)\|_2+\|\Delta\widehat{\mathbf{x}}^h\|+\|\mathbf{x}^h-\mathbf{x}^k\|,\label{first_a}
\end{align}
for all $h\in[k;k+B-1]$.\\

\noindent\textbf{1) Bounding} $a_{i,1}^k$: there holds
\begin{align}
\nonumber&a_{i,1}^k\stackrel{(a)}{\leq}\frac{3\rho L}{2}\Bigg(2\|\mathbf{x}^k-\mathbf{x}^\star(\mathbf{x}^k)\|_2^2+(1+\gamma^2)\|\Delta\widehat{\mathbf{x}}^{l_{i,0}^k}\|_2^2\\
\nonumber&+2\|\mathbf{x}^{l^k_{i,0}}-\mathbf{x}^k\|_2^2\Bigg)\stackrel{(b)}{\leq}3\rho L\Bigg(\kappa^2(1+L+NL_m)^2\Big(\|\Delta\widehat{\mathbf{x}}^k\|_2^2\\
\nonumber&+C_2\sum\limits_{l=k-D}^{k+B-2}\|\Delta\widehat{\mathbf{x}}^l_{i^l}\|_2^2\Big)+(NC_2+1)(1+\gamma^2)\sum\limits_{l=k-D}^{k+B-1}\|\Delta\widehat{\mathbf{x}}^l_{i^l}\|_2^2\\
\label{fifth_a}&+\gamma^2(B-N+1)\sum\limits_{l=k-D}^{k+B-2}\|\Delta\widehat{\mathbf{x}}^l_{i^l}\|_2^2\Bigg)
\stackrel{(c)}{\leq}\rho L\beta_1\sum\limits_{l=k-D}^{k+B-1}\|\Delta\widehat{\mathbf{x}}^l_{i^l}\|_2^2,
\end{align}	
where in (a) we used \eqref{first_a} and the Young's inequality; (b) follows from \eqref{lip}, \eqref{fourth_bis_a}, and the fact that, for any $k\geq0$,
\begin{align}
	\nonumber&\|\mathbf{x}^k-\mathbf{x}^\star(\mathbf{x}^k)\|_2\\\nonumber&\stackrel{B1}{\leq}\kappa\|\mathbf{x}^k-\texttt{prox}_G\left(\nabla_{\mathbf{x}}F(\mathbf{x}^k)-\mathbf{x}^k\right)\|_2\\\nonumber&\stackrel{\eqref{lemma3}}{\leq}\kappa(1+L+NL_m)\|\widehat{\mathbf{x}}(\bar{\mathbf{x}}^k)-\mathbf{x}^k\|_2\\
	&\leq\kappa(1+L+NL_m)\left(\|\widehat{\mathbf{x}}(\bar{\mathbf{x}}^k)-\widehat{\mathbf{x}}^k\|_2+\|\Delta\widehat{\mathbf{x}}^k\|_2\right)\label{errbd};
\end{align}
and in (c) we used  \eqref{fourth_bis_a} and defined\begin{align*}
\beta_1\triangleq& C_2\Bigg(\kappa^2(1+L+NL_m)^2(2N+1)+N(1+\gamma^2)\Bigg)\\&+\kappa^2(1+L+NL_m)^2+1+\gamma^2(B-N+2).
\end{align*}

\noindent\textbf{2) Bounding} $b_{i,t,1}^k$ and $c_{i,1}^k$: for $t\in[1;T^k_i-1]$,
\begin{align}
\nonumber&b_{i,t,1}^k\stackrel{(a)}{\leq}\frac{\rho L}{2}\Big(3\|\mathbf{x}^k-\mathbf{x}^\star(\mathbf{x}^k)\|_2^2+(3+\gamma^2)\|\Delta\widehat{\mathbf{x}}^{l^k_{i,t}}\|_2^2\\
\nonumber&\hspace{-.15cm}+3\|\mathbf{x}^{l^k_{i,t}}-\mathbf{x}^k\|_2^2\Big)\stackrel{(b)}{\leq}\frac{\rho L}{2}\Bigg(6\kappa^2(1+L+NL_m)^2\Bigg(\|\Delta\widehat{\mathbf{x}}^k\|_2^2\\
\nonumber&\hspace{-.15cm}+C_2\sum\limits_{l=k-D}^{k+B-2}\|\Delta\widehat{\mathbf{x}}_{i^l}^l\|_2^2\Bigg)+2(NC_2+1)(3+\gamma^2)\hspace{-.1cm}\sum\limits_{l=k-D}^{k+B-1}\hspace{-.15cm}\|\Delta\widehat{\mathbf{x}}_{i^l}^l\|_2^2\\
\label{fifth_bis_a}&\hspace{-.15cm}+3\gamma^2(B-N+1)\sum\limits_{l=k-D}^{k+B-2}\|\Delta\widehat{\mathbf{x}}_{i^l}^l\|_2^2\Bigg)\stackrel{(c)}{\leq}\rho L\beta_2\hspace{-.1cm}\sum\limits_{l=k-D}^{k+B-1}\hspace{-.15cm}\|\Delta\widehat{\mathbf{x}}_{i^l}^l\|_2^2,
\end{align}
where in (a) we used \eqref{first_a} and the Young's inequality; (b) follows from \eqref{lip}, \eqref{fourth_bis_a}, \eqref{errbd}; and in (c) we used \eqref{fourth_bis_a} and defined\begin{align*}
\beta_2\triangleq& C_2\Bigg(3\kappa^2(1+L+NL_m)^2(2N+1)+N(3+\gamma^2)\Bigg)\\&+6\kappa^2(1+L+NL_m)^2+3+\frac{\gamma^2}{2}(3B-3N+5).
\end{align*}
Following the same steps as in \eqref{fifth_bis_a}, it is not difficult to prove:
\begin{equation}\label{fifth_tris_a}
c_{i,1}^k\leq\rho L \beta_2\sum\limits_{l=k-D}^{k+B-1}\|\Delta\widehat{\mathbf{x}}^l_{i^l}\|^2_2.
\end{equation}

\noindent\textbf{3) Bounding \texttt{term V}} : there holds,
\begin{align}
&\nonumber\texttt{term V}\stackrel{(a)}{\leq}2\|\mathbf{x}^k-\mathbf{x}^\star(\mathbf{x}^k)\|_2^2+2\gamma^2\left\|\Delta\widehat{\mathbf{x}}_i^{l^k_{i,0}}\right\|_2^2\\
\nonumber&+(L_i^2+L^2(\rho-1))\Bigg(\|\Delta\widehat{\mathbf{x}}^{l^k_{i,0}}\|_2^2+D\gamma^2\sum\limits_{l=l^k_{i,0}-D}^{l^k_{i,0}-1}\|\Delta	\widehat{\mathbf{x}}_{i^l}^l\|_2^2\Bigg)\Bigg)\\
\label{eight_a}&\stackrel{(b)}{\leq}\beta_4\sum\limits_{l=k-D}^{k+B-1}\|\Delta\widehat{\mathbf{x}}_{i^l}^l\|_2^2,
\end{align}
where in (a) we used \eqref{x_tilde} and the Young's inequality; and in (b) we used \eqref{lip}, \eqref{fourth_bis_a}, \eqref{errbd}, and defined
\begin{align*}
\beta_4\triangleq&2C_2\Bigg(2\kappa^2(1+L+NL_m)^2(2N+1)\\&+N\left(L_m^2+L^2(\rho-1)\right)\Bigg)+2\kappa^2(1+L+NL_m)^2\\&+\left(L_m^2+L^2(\rho-1)\right)(1+D\gamma^2)+2\gamma^2.
\end{align*}

\noindent\textbf{4) Bounding \texttt{term VI}} : for $t\in[0,T^k_i]$,
\begin{align}
\nonumber&\texttt{term VI}\stackrel{(a)}{\leq}(L^2+L_i^2)\|\mathbf{x}^{l^k_{i,t}}(i,i)-\widehat{\mathbf{x}}^{l^k_{i,t}}\|_2^2+\frac{1}{2}\left\|\Delta\widehat{\mathbf{x}}^{l^k_{i,t}}_i\right\|_2^2\\
\nonumber&\stackrel{\eqref{x_tilde}}{\leq}(L^2+L_i^2)\left(2\left\|\Delta\widehat{\mathbf{x}}^{l^k_{i,t}}\right\|_2^2+2D\gamma^2\sum\limits_{h=l^k_{i,t}-D}^{l^k_{i,t}-1}\|\Delta\widehat{\mathbf{x}}_{i^h}^h\|_2^2\right)\\
\label{sixth_a}&\quad+\frac{1}{2}\left\|\Delta\widehat{\mathbf{x}}^{l^k_{i,t}}_i\right\|_2^2\stackrel{(b)}{\leq}\beta_3\sum\limits_{l=k-D}^{k+B-1}\|\Delta\widehat{\mathbf{x}}_{i^l}^l\|_2^2,
\end{align}
where in (a) we used A2, B2, B3, and the Young's inequality; and in (b) we used \eqref{fourth_bis_a} and defined
\begin{align*}
&\beta_3\triangleq2(L^2+L_m^2)\left(2NC_2+D\gamma^2+1\right)+\frac{1}{2}.
\end{align*}

\noindent\textbf{5) Bounding \texttt{term VII}} : for $t\in[0,T^k_i]$,
\begin{align}
\nonumber&\texttt{term VII}\!\stackrel{(a)}{\leq}\frac{1}{2}\!\!\left(\rho L^2\!\!\sum\limits_{j\in\mathcal{N}_i}\left\|\mathbf{x}^{l^k_{i,t}}-\mathbf{x}^{l^k_{i,t}}(i,j)\right\|_2^2+\!\left\|\Delta\widehat{\mathbf{x}}^{l^k_{i,t}}_i\right\|_2^2\right)\\
\nonumber&\stackrel{\eqref{x_tilde}}{\leq}\frac{1}{2}\left(\rho^2L^2D^2\gamma^2\sum\limits_{l=l^k_{i,t}-D}^{l^k_{i,t}-1}\|\Delta\widehat{\mathbf{x}}_{i^l}^l\|_2^2+\left\|\Delta\widehat{\mathbf{x}}^{l^k_{i,t}}_i\right\|_2^2\right)\\
\label{seventh_a}&\leq\frac{\rho^2L^2D^2\gamma^2+1}{2}\sum\limits_{l=k-D}^{k+B-2}\|\Delta\widehat{\mathbf{x}}_{i^l}^l\|_2^2,
\end{align}
where in (a) we used A2 and the Young's inequality.

\noindent\textbf{6) Bounding \texttt{term VIII} and \texttt{term IX}} :  for $t\in[1,T_i^k-1]$
\begin{align}
\nonumber&\texttt{term VIII}\!\stackrel{(a)}{\leq}\gamma^2\left\|\Delta\widehat{\mathbf{x}}_i^{l^k_{i,t}}\right\|_2^2\!\!+(L_i^2+L^2(\rho-1))\Bigg(\left\|\Delta\widehat{\mathbf{x}}^{l^k_{i,t}}\right\|_2^2\\
\label{ninth_a}&+D\gamma^2\sum\limits_{l=l^k_{i,t}-D}^{l^k_{i,t}-1}\|\Delta	\widehat{\mathbf{x}}_{i^l}^l\|_2^2\Bigg)\Bigg)\stackrel{(b)}{\leq}\beta_5\sum\limits_{l=k-D}^{k+B-1}\|\Delta\widehat{\mathbf{x}}_{i^l}^l\|_2^2,
\end{align}
where in (a) we used \eqref{x_tilde} and the Young's inequality; and in (b) we used \eqref{fourth_bis_a}, and defined
\begin{align*}
&\beta_5\triangleq\left(L_m^2+L^2(\rho-1)\right)\left(2NC_2+D\gamma^2+2\right)+\gamma^2.
\end{align*}
As done in \eqref{ninth_a}, it is easy to prove that
\begin{align}
&\texttt{term IX}\leq\beta_5\sum\limits_{l=k-D}^{k+B-1}\|\Delta\widehat{\mathbf{x}}_{i^l}^l\|_2^2.\label{tenth_a}
\end{align}

Using the above results, we can bound $a_{i,4}^k$, $b_{i,t,4}^k$, and $c_{i,4}^k$. According to definition of $a_{i,4}^k$, we have
\begin{equation}
a_{i,4}^k\stackrel{\eqref{fifth_a}, \eqref{sixth_a}-\eqref{eight_a}}{\leq}\alpha_1\sum\limits_{l=k-D}^{k+B-1}\|\Delta\widehat{\mathbf{x}}_{i^l}^l\|_2^2,\label{aik}
\end{equation}
where
\begin{align}
&\alpha_1\hspace{-.1cm}\triangleq\left((1-\gamma)\left(\beta_3+\frac{L\gamma(\rho^2LD^2\gamma+1)+1}{2}\right)+\rho L\beta_1+\beta_4\right)\hspace{-.1cm}.\label{alpha1}
\end{align}

For $b_{i,t,4}^k$ and $c_{i,4}^k$, we have: $t\in[1,T^k_i-1]$,
\begin{equation}
b^k_{i,i,4};c_{i,4}^k\stackrel{\eqref{fifth_bis_a}-\eqref{seventh_a}, \eqref{ninth_a},  \eqref{tenth_a}}{\leq}\alpha_2\sum\limits_{l=k-D}^{k+B-1}\|\Delta\widehat{\mathbf{x}}_{i^l}^l\|_2^2,\label{bcik}
\end{equation}
where
\begin{align}
&\alpha_2\triangleq\left((1-\gamma)\left(\beta_3+\frac{L\gamma(\rho^2LD^2\gamma+1)+1}{2}\right)+\rho L\beta_2+\beta_5\right).\label{alpha2}
\end{align}

 Combining \eqref{begin_a}, \eqref{one_a}, \eqref{two_a}, \eqref{three_a}, \eqref{aik}, and \eqref{bcik} yields:
\begin{align}
\nonumber&V(\mathbf{x}^{k+B})-V(\mathbf{x}^\star(\mathbf{x}^k))\leq\frac{1-\gamma}{\gamma}\left(V(\mathbf{x}^k)-V(\mathbf{x}^{k+B})\right)\\
\nonumber&+\sum\limits_{i=1}^N\Bigg(\gamma\left(g_i\left(\mathbf{x}_i^{l^k_{i,T^k_i}}\right)-g_i\left(\widehat{\mathbf{x}}_i^{l^k_{i,T^k_i}}\right)\right)\\
\nonumber&+\gamma\sum\limits_{t=1}^{T^k_i-1}\left(g_i\left(\mathbf{x}_i^{l^k_{i,t}}\right)-g_i\left(\widehat{\mathbf{x}}_i^{l^k_{i,t}}\right)\right)+(\gamma-1)g_i\left(\mathbf{x}_i^{l^k_{i,0}}\right)\\
\nonumber&-\gamma g_i\left(\widehat{\mathbf{x}}_i^{l^k_{i,0}}\right)+g_i\left(\mathbf{x}_i^{k+B}\right)\Bigg)+(N\alpha_1\\
\nonumber&+(B-N)\alpha_2)\sum\limits_{l=k-D}^{k+B-1}\|\Delta\widehat{\mathbf{x}}_{i^l}^l\|_2^2\stackrel{A3}{\leq}\frac{1-\gamma}{\gamma}\big(V(\mathbf{x}^k)\\
\nonumber&-V(\mathbf{x}^{k+B})\big)+\sum\limits_{i=1}^N\Bigg(\gamma\left(g_i\left(\mathbf{x}_i^{l^k_{i,T^k_i}}\right)-g_i\left(\widehat{\mathbf{x}}_i^{l^k_{i,T^k_i}}\right)\right)\\
\nonumber&+\gamma\sum\limits_{t=1}^{T^k_i-1}\left(g_i\left(\mathbf{x}_i^{l^k_{i,t}}\right)-g_i\left(\widehat{\mathbf{x}}_i^{l^k_{i,t}}\right)\right)+(\gamma-1)g_i\left(\mathbf{x}_i^{l^k_{i,0}}\right)\\
\nonumber&-\gamma g_i\left(\widehat{\mathbf{x}}_i^{l^k_{i,0}}\right)+(1-\gamma)g_i\left(\mathbf{x}_i^{l^k_{i,T^k_i}}\right)+\gamma g_i\left(\widehat{\mathbf{x}}_i^{l^k_{i,T^k_i}}\right)\Bigg)\\
\nonumber&+\left(N\alpha_1+(B-N)\alpha_2)\right)\sum\limits_{l=k-D}^{k+B-1}\|\Delta\widehat{\mathbf{x}}_{i^l}^l\|_2^2\\
\nonumber&=\frac{1-\gamma}{\gamma}\left(V(\mathbf{x}^k)-V(\mathbf{x}^{k+B})\right)+\sum\limits_{i=1}^N\Bigg(g_i\left(\mathbf{x}_i^{l^k_{i,T^k_i}}\right)\\
\nonumber&\quad+\gamma\sum\limits_{t=1}^{T^k_i-1}\left(g_i\left(\mathbf{x}_i^{l^k_{i,t}}\right)-g_i\left(\widehat{\mathbf{x}}_i^{l^k_{i,t}}\right)\right)\\
&\nonumber\quad+(\gamma-1)g_i\left(\mathbf{x}_i^{l^k_{i,0}}\right)-\gamma g_i\left(\widehat{\mathbf{x}}_i^{l^k_{i,0}}\right)+(N\alpha_1\\
\nonumber&\quad+(B-N)\alpha_2)\sum\limits_{l=k-D}^{k+B-1}\|\Delta\widehat{\mathbf{x}}_{i^l}^l\|_2^2\stackrel{A3}{\leq}\frac{1-\gamma}{\gamma}\big(V(\mathbf{x}^k)\\
\nonumber&-V(\mathbf{x}^{k+B})\big)+\sum\limits_{i=1}^N\Bigg((1-\gamma)g_i\left(\mathbf{x}_i^{l^k_{i,T^k_i-1}}\right)+\gamma g_i\left(\widehat{\mathbf{x}}_i^{l^k_{i,T^k_i-1}}\right)\\
\nonumber&+\gamma\sum\limits_{t=1}^{T^k_i-1}\left(g_i\left(\mathbf{x}_i^{l^k_{i,t}}\right)-g_i\left(\widehat{\mathbf{x}}_i^{l^k_{i,t}}\right)\right)\\
\nonumber&+(\gamma-1)g_i\left(\mathbf{x}_i^{l^k_{i,0}}\right)-\gamma g_i\left(\widehat{\mathbf{x}}_i^{l^k_{i,0}}\right)\\
\nonumber&+\left(N\alpha_1+(B-N)\alpha_2\right)\sum\limits_{l=k-D}^{k+B-1}\|\Delta\widehat{\mathbf{x}}_{i^l}^l\|_2^2\\
\nonumber&=\frac{1-\gamma}{\gamma}\left(V(\mathbf{x}^k)-V(\mathbf{x}^{k+B})\right)+\sum\limits_{i=1}^N\Bigg(g_i\left(\mathbf{x}_i^{l^k_{i,T^k_i-1}}\right)\\
\nonumber&\quad+\gamma\sum\limits_{t=1}^{T^k_i-2}\left(g_i\left(\mathbf{x}_i^{l^k_{i,t}}\right)-g_i\left(\widehat{\mathbf{x}}_i^{l^k_{i,t}}\right)\right)\\
\nonumber&\quad+(\gamma-1)g_i\left(\mathbf{x}_i^{l^k_{i,0}}\right)-\gamma g_i\left(\widehat{\mathbf{x}}_i^{l^k_{i,0}}\right)\\
\nonumber&\quad+\left(N\alpha_1+(B-N)\alpha_2\right)\sum\limits_{l=k-D}^{k+B-1}\|\Delta\widehat{\mathbf{x}}_{i^l}^l\|_2^2\leq\frac{1-\gamma}{\gamma}\big(V(\mathbf{x}^k)\\
&-V(\mathbf{x}^{k+B})\big)+\left(N\alpha_1+(B-N)\alpha_2\right)\sum\limits_{l=k-D}^{k+B-1}\|\Delta\widehat{\mathbf{x}}_{i^l}^l\|_2^2.
\end{align}
\bibliographystyle{IEEEtran}
\bibliography{biblio}

\vspace{-1.3cm}

\begin{IEEEbiography}[{\includegraphics[width=1in,height=1.25in,clip,keepaspectratio]{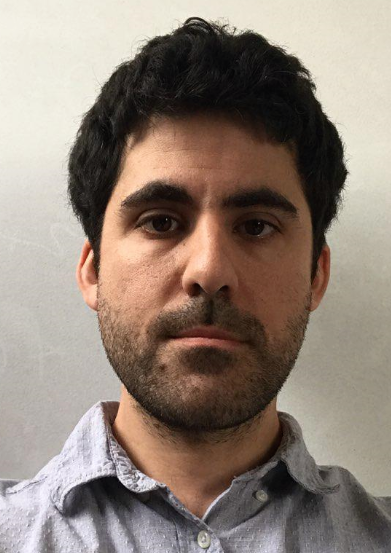}}]{Loris Cannelli} received his B.S. and M.S. in electrical and	telecommunication engineering from the University of Perugia, Italy, his M.S. in electrical engineering from State	University of New York at Buffalo, NY, and his Ph.D. in industrial engineering from the Purdue University, West Lafayette, IN, USA. His research interests include optimization algorithms, machine learning, and big-data analytics.\vspace{-1.4cm}
\end{IEEEbiography}

\begin{IEEEbiography}[{\includegraphics[width=1in,height=1.25in,clip,keepaspectratio]{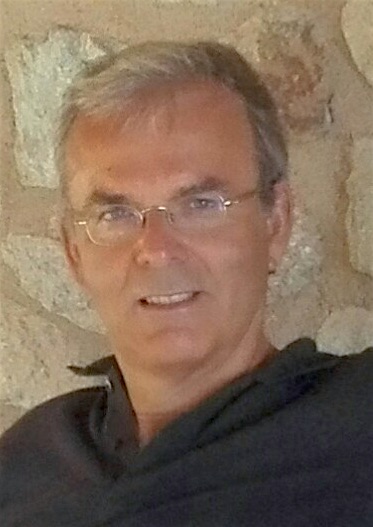}}]{Francisco Facchinei} received the Ph.D. degree in	system engineering from the University of Rome, ``La Sapienza,'' Rome, Italy. He is a Full Professor of	operations research, Engineering Faculty, University	of Rome, ``La Sapienza.'' His research interests focus
	on theoretical and algorithmic issues related to nonlinear optimization, variational inequalities, complementarity problems, equilibrium programming, and	computational game theory.\vspace{-1.4cm}
\end{IEEEbiography}

\begin{IEEEbiography}[{\includegraphics[width=1in,height=1.25in,clip,keepaspectratio]{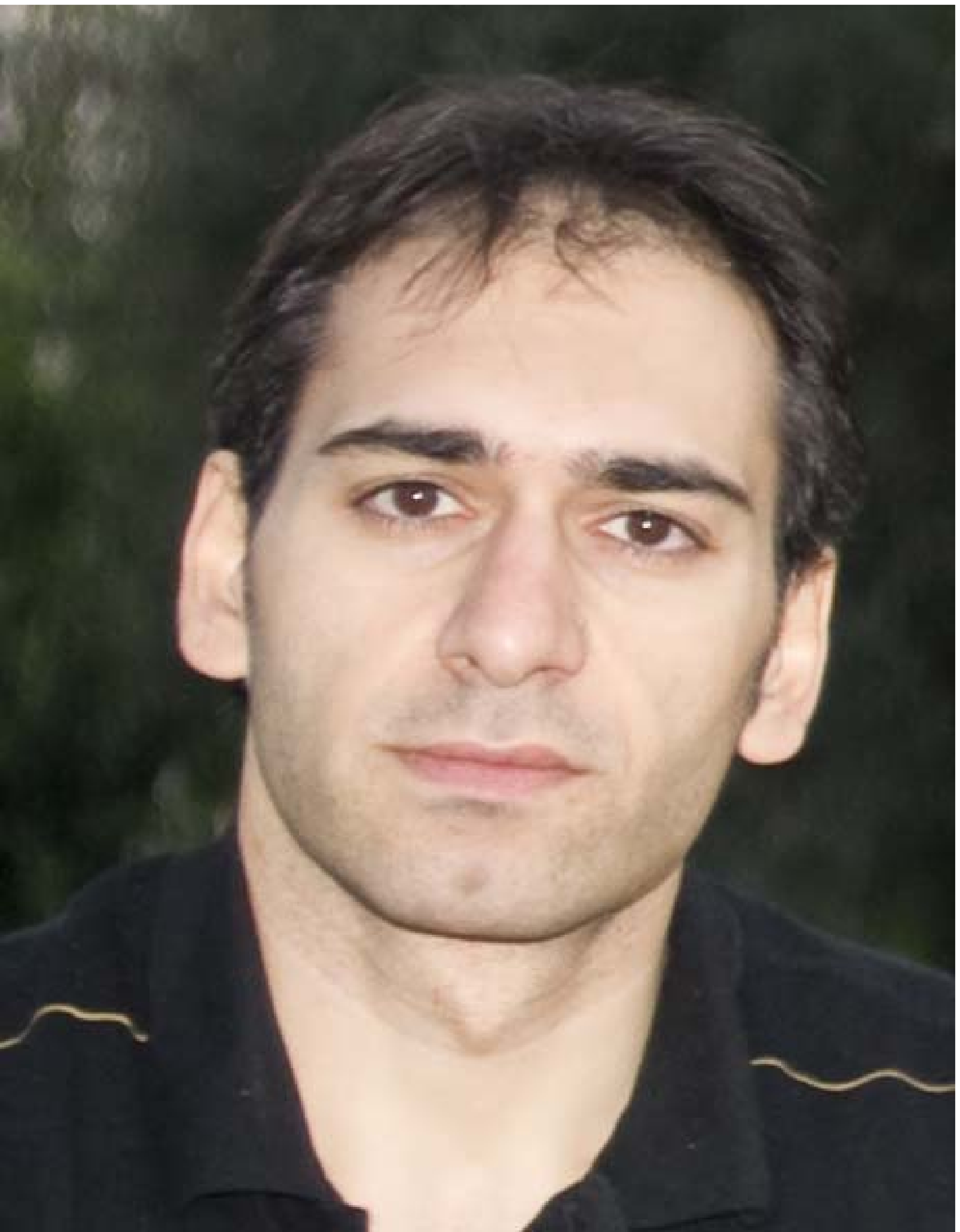}}]{Gesualdo Scutari}
	(S'05-M'06-SM'11) received the Electrical Engineering and Ph.D. degrees (both with honors)
	from the University of Rome ``La Sapienza,'' Rome, Italy, in 2001 and 2005, respectively.
	He is the Thomas and Jane Schmidt Rising Star Associate Professor  with the School of Industrial Engineering, Purdue University, West Lafayette, IN, USA. 
	He had previously held several research appointments, namely, at the University of California at
	Berkeley, Berkeley, CA, USA; Hong Kong University of Science and Technology, Hong Kong;
	and University of Illinois at Urbana-Champaign,
	Urbana, IL, USA.  His research interests include continuous and distributed optimization,   equilibrium programming, and their applications to signal processing and
	machine learning.
	He is a Senior Area Editor of the IEEE Transactions On Signal Processing and an Associate Editor of the IEEE Transactions on Signal and Information Processing over Networks. 
	He served on the IEEE Signal Processing Society Technical Committee on Signal Processing for Communications (SPCOM).
	He was the recipient of the 2006 Best Student Paper Award at the IEEE ICASSP 2006, the 2013 NSF  CAREER Award,  the 2015 Anna Maria Molteni Award for Mathematics and Physics, and the 2015 IEEE Signal Processing Society Young Author Best Paper Award.\vspace{-1.4cm}
\end{IEEEbiography}

\begin{IEEEbiography}[{\includegraphics[width=1in,height=1.25in,clip,keepaspectratio]{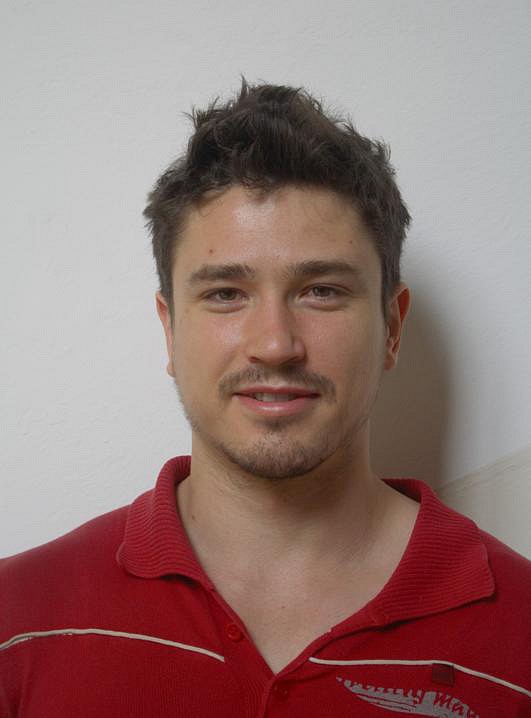}}]{Vyacheslav Kungurtsev} received his B.S. in Mathematics from Duke University in 2007, and his PhD in Mathematics with a specialization in Computational Science from the University of California - San Diego, in 2013.
	He spent one year as postdoctoral researcher at KU Leuven for the Optimization for Engineering Center, and since 2014 he has been a Researcher at Czech Technical University in Prague working on various aspects of continuous optimization.
\end{IEEEbiography}
\end{document}